\documentclass[preprint,a4paper,10pt]{article}
\title{On the multi-Koszul property for connected algebras}
\author{Estanislao Herscovich
\footnote{Institut Joseph Fourier, Universit\'e Grenoble I, France, and Departamento de Matem\'atica, FCEyN, Universidad de Buenos Aires, Argentina.}
}
\date{}

\usepackage{amsmath,amsthm,amsfonts,amssymb,stmaryrd}
\usepackage{latexsym}
\usepackage{fullpage}
\usepackage{a4,color,palatino,fancyhdr}
\usepackage{graphicx}
\usepackage{float}  
\usepackage[small, bf, margin=90pt, tableposition=bottom]{caption}
\RequirePackage{amssymb}
\RequirePackage[T1]{fontenc}
\usepackage[breaklinks, naturalnames]{hyperref}
\usepackage{amsrefs}

\input{xy}
\xyoption{all}
\xyoption{poly}
\usepackage[all]{xy}

\newtheorem{theorem}{Theorem}[section]
\newtheorem{proposition}[theorem]{Proposition}
\newtheorem{definition}[theorem]{Definition}
\newtheorem{lemma}[theorem]{Lemma}
\newtheorem{corollary}[theorem]{Corollary}
\newtheorem{remark}[theorem]{Remark}
\newtheorem{example}[theorem]{Example}
\newtheorem{fact}[theorem]{Fact}

\numberwithin{equation}{section}

\def\cl#1{{\langle #1 \rangle}}

\def\place{{-}}

\def\Ker{\mathop{\rm Ker}\nolimits}

\def\Tor{\mathop{\rm Tor}\nolimits}
\def\Ext{\mathop{\rm Ext}\nolimits}

\def\Hom{{\mathrm {Hom}}}

\def\grMod{{\mathrm {grMod}}}

\newcommand\ZZ{{\mathbb{Z}}}
\newcommand\NN{{\mathbb{N}}}

\def\place{{-}}

\begin{document}

\maketitle
                                                     
\hrulefill

\begin{abstract}
In this article we introduce the notion of \emph{multi-Koszul algebra} for the case of a locally finite dimensional nonnegatively graded connected algebra, 
as a generalization of the notion of (generalized) Koszul algebras defined by R. Berger for homogeneous algebras, 
which were in turn an extension of Koszul algebras introduced by S. Priddy. 
This notion also extends and generalizes the one recently introduced by the author and A. Rey, which was for the particular case of algebras further assumed to be finitely generated in degree $1$ 
and with a finite dimensional space of relations. 
The idea of this new notion in the realm of arbitrary locally finite dimensional nonnegatively graded connected algebras, 
which should be perhaps considered as a probably interesting common property for several of these algebras, was to find a grading independent description of some of the interesting 
properties shared by all generalized Koszul algebras. 
In order to simplify the exposition and unify the left and right possible sides of the multi-Koszul property, the very definition of multi-Koszul algebras starts 
by considering the minimal graded projective resolution of the algebra $A$ as an $A$-bimodule, which may be used to compute the corresponding Hochschild (co)homology groups. 
This new definition includes several new interesting examples, \textit{e.g.} the super Yang-Mills algebras introduced by M. Movshev and A. Schwarz, which are not generalized Koszul 
or even multi-Koszul for the previous definition given by the author and Rey in any natural manner.
On the other hand, we provide an equivalent description of the new definition in terms of the \textrm{Tor} (or \textrm{Ext}) groups, similar to the existing one for homogeneous algebras, 
and we show that several of the typical homological computations performed for the generalized Koszul algebras are also possible in this more general setting. 
In particular, we give a very explicit description of the $A_{\infty}$-algebra structure of the Yoneda algebra of a multi-Koszul algebra, 
which has a similar pattern as for the case of generalized Koszul algebras. 
We also show that the Yoneda algebra of a finitely generated multi-Koszul algebra with a finite dimensional space of relations is generated in degrees $1$ and $2$, 
so a $\mathcal{K}_{2}$ algebra in the sense of T. Cassidy and B. Shelton. 
\end{abstract}



\hrulefill

\section{\texorpdfstring{Introduction}{Introduction}}
\label{sec:int}

Koszul algebras were introduced by S. Priddy in \cite{P}, motivated by the article \cite{K} of J.-L. Koszul. 
They seem to have a pervasive appearance in representation theory (\textit{cf.} \cite{BGS1}, \cite{BGS2}), algebraic geometry (\textit{cf.} \cite{F}), quantum groups (\textit{cf.} \cite{M}), 
and combinatorics (\textit{cf.} \cite{HL}), to mention a few. 
These algebras are necessarily quadratic, \textit{i.e.} of the form $T(V)/\cl{R}$, with $R \subseteq V^{\otimes 2}$. 
R. Berger generalized in \cite{B2} the notion of Koszul algebras (\textit{cf.} also \cite{GMMZ}) for the case of homogeneous algebras, 
\textit{i.e.} algebras given by $T(V)/\cl{R}$, with $R \subseteq V^{\otimes N}$, for $N \geq 2$. 
They were called \emph{generalized Koszul}, or \emph{$N$-Koszul} if the mention to the degree of the relations was to be indicated, and the case $N = 2$ of the definition introduced by Berger 
coincides with the one given by Priddy. 
The general definition shares a lot of good properties with the one given by Priddy, justifying the terminology (see for example \cites{B2, BG}). 
In particular, the Yoneda algebra of an $N$-Koszul algebra is finitely generated (in degrees $1$ and $2$), and its structure is easily computed from that of the original algebra. 
We would like to point out that the new class of algebras satisfying the Koszul property of Berger lacks however of other interesting properties, 
\textit{e.g.} they are not closed under taking duals, or under considering graded Ore extensions, the Yoneda algebra of an $N$-Koszul algebra is not formal for $N \geq 3$, etc. 

Nevertheless, it is still a natural question to ask if there exists an analogous definition for much general types of algebras, which satisfy some of the of the good homological properties 
satisfied by generalized Koszul algebras. 
In fact, in the recent paper \cite{HeRe}, A. Rey and the author have considered an extension of the generalized Koszul property, called \emph{multi-Koszul}, to the case of finitely generated 
nonnegatively graded connected algebras generated in degree $1$ and with a finite dimensional space of relations, which coincides with the definition given by Berger if the algebra is homogeneous. 
Even though the definition of that article has several advantages, for it satisfied several homological properties as for the case of generalized Koszul algebras, 
it seemed to be a little \emph{ad hoc}, and also too restrictive, for it cannot be applied in much general contexts, where the previous assumptions on the algebras do not hold 
(\textit{e.g.} if the algebra is not generated in degree $1$). 
In this article we proceed in a completely different manner to get rid of the hypotheses stated before and to consider all locally finite dimensional nonnegatively graded connected algebras. 
In fact, our main goal is to provide a collection of algebras which have a very similar homological behaviour to one of the generalized Koszul algebras, 
for which the Yoneda algebra structure (and even the $A_{\infty}$-algebra structure) is directly deduced from that of the original algebra, 
in a similar fashion as the case of generalized Koszul algebras (see Theorem \ref{theorem:yoneda}, Proposition \ref{proposition:yonedaainf}, and Theorem \ref{theorem:ainftres}), 
and the former is in fact finitely generated if the algebra in question is finitely generated with a finite dimensional space of relations (see Proposition \ref{proposition:k2}). 
Furthermore, our construction of the multi-Koszul complex, and thus our definition of multi-Koszul algebra, is independent of the (nonnegative) grading of the algebra, 
and it restricts to the notion introduced in \cite{HeRe} for the additional assumption that the algebra 
is finitely generated in degree $1$ and it has a finite dimensional space of relations (see Proposition \ref{proposition:equiv}), 
which further implies that it coincides with the definition given by Berger for homogeneous algebras, and thus with the one given by Priddy for quadratic algebras. 
All these properties have precisely given shape to the definition, which in some sense consists of requiring that the \textrm{Tor} (or \textrm{Ext}) groups should have an 
``optimal'' behaviour and relations between them, in such a way that the standard homological techniques of generalized Koszul algebras still hold. 
This is somehow the main result of the present work and it gives a reason for this new definition, contrary to the one given in \cite{HeRe}, which could have seemed to be somehow arbitrary. 

We would like to point out that even tough this new definition could still seem to be rather restrictive, it allows several interesting examples, 
for it includes the super Yang-Mills algebras introduced by M. Movshev and A. Schwarz in \cite{MS06} 
and further studied by the author in \cite{Her10}, which were one of the main motivations for the present work (see Example \ref{example:sym}). 
These algebras are not generated in degree $1$, so they cannot be multi-Koszul in the sense defined in \cite{HeRe}, and in particular they cannot be generalized Koszul either.
We also admit that our definition is probably not the most general possible and reasonable extension of the Koszul property for all locally finite dimensional nonnegatively graded connected algebras, but all the interesting properties mentioned in the previous paragraph make us believe that any sensible 
such general definition of \textit{Koszul-like} algebra in this general context, if it exists, should necessarily include our definition as a special case. 
That is also one of the main reasons why we have refrained from calling this new family of algebras Koszul.

The contents of the article are as follows. 
We start by recalling in Section \ref{sec:homo} several well-known definitions and results about the category of graded modules over a locally finite dimensional 
nonnegatively graded connected algebra.

Section \ref{sec:multikoszul} is devoted to the definition of multi-Koszul algebras and to prove some properties for this family of algebras. 
In order to simplify the exposition and unify the left and right sides, it seems useful to work directly with the minimal projective resolution of the algebra $A$ as an $A$-bimodule, 
and to derive from it the minimal projective resolutions of the left or right trivial $A$-module $k$. 
The definition requires however several preliminary results about the spaces that will provide the generators of these minimal projective resolutions, 
which are in given in Subsection \ref{subsec:inter} and in the beginning of Subsection \ref{subsec:defimultikoszul}. 
The first main result, Proposition \ref{proposition:tresimp} (see also Proposition \ref{proposition:tresimpext}) gives a (co)homological description of multi-Koszul algebras in term of their \textrm{Tor} 
(or \cal{E}xt) groups, which yields a left-right symmetry of the definition.
Moreover, in Subsection \ref{subsec:propmultikoszul} we also provide several properties satisfied by these algebras, and in particular we show that they are stable under free 
products (see Proposition \ref{proposition:libre}) and they are $\mathcal{K}_2$ algebras, in the sense defined by B. Cassidy and T. Shelton in \cite{CS}, 
if the algebras are further assumed to be finitely generated with finite dimensional spaces of relations (see Proposition \ref{proposition:k2}). 

Finally, in Section \ref{sec:yonedamultikoszul} we provide an explicit description of the algebra structure, and further of $A_{\infty}$-algebra structure, of the Yoneda algebra of a multi-Koszul algebra (see Theorem \ref{theorem:yoneda}, Proposition \ref{proposition:yonedaainf}, and Theorem \ref{theorem:ainftres}). 
In particular, if the multi-Koszul algebra is finitely generated generated in degree $1$ and with a finite number of relations, this description coincides with the $A_{\infty}$-algebra 
structure of the Yoneda algebra derived from Prop. 3.21 in \cite{HeRe}, as explained in Rmk. 3.22, and from Rmk. 3.25 of the same article. 
The proofs we give here in this more general setting are however completely different. 
The characterization of both structures given in the previous statements follows a similar pattern to the one given for generalized Koszul algebras in \cite{BM}, Prop. 3.1, and \cite{HeL}, Thm. 6.5, 
which gives another indication that the homological behaviour of the algebras satisfying this new definition is parallel to the one of generalized Koszul algebras. 

Throughout this article $k$ will denote a field, and all vector spaces will be over $k$. 
Moreover, $V$ will always be a vector space, and $A$ a nonnegatively graded connected (associative) algebra over $k$ (with unit), 
to which we will usually just refer as an algebra, with irrelevant ideal $A_{>0} = \bigoplus_{n>0} A_{n}$. 
All unadorned tensor products $\otimes$ will be considered over $k$, unless otherwise stated. 
We shall typically denote the vector space $V^{\otimes n}$ by $V^{(n)}$, for $n \in \ZZ$, where we follow the convention $V^{(n)} = 0$, if $n < 0$, and $V^{(0)} = k$, 
and an elementary tensor $v_{1} \otimes \dots \otimes v_{n} \in V^{(n)}$ will be usually written by $v_{1} \dots v_{n}$. 
However, given (possibly graded) vector subspaces $T_{1}, \dots, T_{m}$ of $T(V)$, and elements $t_{i} \in T_{i}$, for $i = 1, \dots, m$, 
we may denote the element $t_{1} \otimes \dots \otimes t_{m} \in T_{1} \otimes \dots \otimes T_{m} \subseteq T(V)$ by $t_{1} |\dots |t_{m}$, to avoid confusion. 
Finally, given two chain complexes $(C_{\bullet},d_{\bullet})_{\bullet \in \ZZ}$ and $(C'_{\bullet},d'_{\bullet})_{\bullet \in \ZZ}$ of modules over a ring, we say that both complexes 
\emph{coincide up to homological degree} $m \in \ZZ$ if $C_{n} = C'_{n}$ and $d_{n} = d'_{n}$, for all $n \leq m$.   

\section{\texorpdfstring{Preliminaries on minimal resolutions of graded algebras}{Preliminaries on minimal resolutions of graded algebras}}
\label{sec:homo}

In this section we shall recall some basic facts about the category of (bounded below) graded modules over a nonnegatively graded connected algebra. 
We refer to \cite{C}, Exp. 15, or \cite{B4} for all the proofs of the mentioned results. 
Even though we impose the assumption that $A$ is locally finite dimensional, most of the results of this section hold without that hypothesis.  

We recall that a \emph{graded vector space} is a vector space $W$ provided with a direct sum decomposition of the form $W = \oplus_{n \in \ZZ} W_{m}$. 
A nonzero element $w$ in $W_{m}$ is called \emph{homogeneous of degree $m$}, and we will write $|w| = m$.
A \emph{morphism of graded vector spaces of degree $d$}  from $W$ to $W'$ is just a linear map $f : W \rightarrow W'$ such that $f(W_{m}) \subseteq W'_{m+d}$, for all $m \in \ZZ$. 
If we omit the degree of the morphism it will be understood that $d = 0$. 
A \emph{bigraded (co)homological vector space} is a vector space $W$ together with a direct sum decomposition given by $W = \oplus_{(n,m) \in \ZZ^{2}} W_{n,m}$ (resp., $W = \oplus_{(n,m) \in \ZZ^{2}} W^{n}_{m}$. 
We may consider a graded vector space $W = \oplus_{m \in \ZZ} W_{m}$ as a (co)homological bigraded vector space concentrated in (co)homological degree $0$, 
\textit{i.e.} $W_{n,m} = W_{m}$ (resp., $W^{n}_{m} = W_{m}$) if $n = 0$, and $W_{n,m}$ (resp., $W^{n}_{m}$) vanishes otherwise. 
Given a nonzero element $w \in W_{n,m}$ (resp., $w \in W^{n}_{m}$) we will say that it has \emph{(co)homological degree} $n$ and \emph{Adams degree} $m$ 
(this terminology comes from algebraic topology, see for instance \cite{KM}, p. 38), or, simply, \emph{bidegree} $(n,m)$. 
The first will be usually denoted by $\deg(w)$, whereas the second one will be written as $|w|$. 
A \emph{morphism of (co)homological bigraded vector spaces of (co)homological degree $d$ and Adams degree $d'$ (or \emph{bidegree} $(d,d')$)} from $W$ to $W'$ is just a linear map 
$f : W \rightarrow W'$ such that $f(W_{n,m}) \subseteq W'_{n+d,m+d'}$, for all $n, m \in \ZZ$. 
If we omit the degree of the morphism it will be understood that $d = d'= 0$. 
If we talk about the degree of an element or a map, it is implicitly assumed that it is homogeneous. 
If $W = \oplus_{(n,m) \in \ZZ^{2}} W_{n,m}$ is a homological bigraded vector space, it will be also considered as a cohomological bigraded vector space where $W^{n}_{m} = W_{-n,m}$, 
for all $n, m \in \ZZ$, and vice versa. 
By abuse of notation we may thus say that $w \in W_{n,m}$ has cohomological degree $-n$. 

If $W$ is a graded vector space, we may consider the graded vector space $W^{\#}$, called the \emph{graded dual}, 
which has $n$-th homogeneous component $(W_{-n})^{*}$, where $(\place)^{*}$ denotes the usual dual vector space operation. 
We will also consider the analogous graded dual construction $(\place)^{\#}$ in the category of (co)homological bigraded vector spaces, where in that case $W^{\#}$ has component of 
(co)homological degree $n$ and Adams degree $m$ given by $(W_{-n,-m})^{*}$ (resp., $(W^{-n}_{-m})^{*}$). 
If $f : W \rightarrow W'$ is morphism of graded vector spaces of degree $d$, we get another morphism $f^{\#}$ of graded vector spaces of the same degree whose restriction to $(W'_{-n})^{*}$ 
is given by $(f|_{W_{-n-d}})^{*}$. 
If $f : W \rightarrow W'$ is morphism of (co)homological bigraded vector spaces of (co)homological degree $d$, define the morphism $f^{\#}$ of (co)homological bigraded vector spaces of the same (co)homological degree given by $f^{\#} (\lambda) = (-1)^{d \deg(\lambda)} \lambda \circ f$, for $\lambda$ homogeneous.
Note that $(g \circ f)^{\#} = f^{\#} \circ g^{\#}$ if $f, g$ are morphisms of graded vector spaces such that the composition makes sense, 
but $(g \circ f)^{\#} = (-1)^{\deg(f) \deg{g}} f^{\#} \circ g^{\#}$, for $f, g$ morphisms of (co)homological bigraded vector spaces such that the composition makes sense.
We will thus apply the Koszul sign rule to the (co)homological degree but not to the Adams degree. 
This means in particular that, if $W_{1}, \dots, W_{n}$ are locally finite dimensional (co)homological bigraded vector spaces, 
there is an isomorphism $c_{W_{1}, \dots, W_{n}} : W_{1}^{\#} \otimes \dots \otimes W_{n}^{\#} \rightarrow (W_{1} \otimes \dots \otimes W_{n})^{\#}$ 
of the form 
\[     c_{W_{1}, \dots, W_{n}}(f_{1} \otimes \dots \otimes f_{n})(w_{1} \otimes \dots \otimes w_{n}) = (-1)^{s} f_{1}(w_{1}) \dots f_{n}(w_{n}),     \] 
where $s = \sum_{i=1}^{n-1} \deg(w_{i}) (\deg(f_{i+1}) + \dots \deg(f_{n}))$. 
However, if we refrain from applying the Koszul sign rule to any degree whatsoever of a homological bigraded vector space, we will just call it \emph{bigraded vector space}. 

From now on, $A$ will always denote a nonnegatively graded connected algebra with unit $1$. 
We shall further assume that the underlying graded vector space of $A$ is locally finite dimensional, \textit{i.e.} $A = \oplus_{n \in \NN_{0}} A_{n}$, where 
$\mathrm{dim}_{k}(A_{n}) < \infty$, for all $n \in \NN_{0}$, and we shall follow the typical convention $A_{n} = 0$, for $n < 0$. 
This means that there exists a locally finite dimensional positively graded vector space $V$ 
and a surjective morphism of graded algebras of the form $\pi : T(V) \rightarrow A$, 
so $A \simeq T(V)/I$, where $I \subseteq T(V)$ is a homogeneous ideal of $T(V)$. 
Note that for any graded vector space $V$ the tensor algebra $T(V) = \oplus_{N \in \NN_{0}} V^{(N)}$ is provided with two compatible gradings, 
other grading coming from the previous direct sum decomposition, which we shall call \emph{special} or \emph{tensor}, and whose 
$N$-th homogeneous component will be also denoted by $T(V)^{N}$, and another one coming from the grading on $V$, which we call \emph{usual} 
and whose $n$-th homogeneous components will be denoted by $T(V)_{n}$. 
To avoid redundancy we will always assume that the vector space $V$ is canonically isomorphic to $A_{>0}/(A_{>0} . A_{>0})$ (as graded vector spaces). 
In this case we further have that $I \subseteq T(V)^{\geq 2} = T(V)_{>0} . T(V)_{>0}$. 
Let us denote by $R$ a \emph{space of relations} of $A$, \textit{i.e.} a graded vector subspace of $I$ which is isomorphic to 
$I/(T(V)_{>0} . I + I . T(V)_{>0})$ under the canonical projection. 
It can be equivalently defined as a graded vector subspace $R \subseteq I$ satisfying that the ideal $\cl{R}$ of $T(V)$ generated by $R$ coincides with $I$ and that 
\begin{equation}
\label{eq:sprel}
      R \cap (T(V)_{>0} \otimes R \otimes T(V) + T(V) \otimes R \otimes T(V)_{>0}) = 0. 
\end{equation}
Notice that the Hilbert series of $R$ is thus uniquely determined, and the same holds for its first nonvanishing homogeneous component. 
We may thus suppose that $A = T(V)/\cl{R}$, for $R \subseteq T(V)^{\geq 2}$ a graded vector subspace for the usual grading. 
Note that the assumption that $A$ is locally finite dimensional implies that $V$ (hence $T(V)$) is a locally finite dimensional nonnegatively graded vector space, 
and then $R$ is also so. 

A \emph{graded left (resp., right) $A$-module} $M=\bigoplus_{n \in \ZZ} M_{n}$ is a graded vector space together with a left (resp., right) action of $A$ on $M$ 
such that $A_{m} M_{n} \subseteq M_{m+n}$ (resp., $M_{m} A_{n} \subseteq M_{m+n}$), and we shall sometimes refer to them simply as \emph{left (resp., right) $A$-modules}. 
Moreover, in the rest of the section we will mostly deal with left modules, which may be thus called just \emph{modules} (or \emph{graded modules}), 
even though all the considerations apply verbatim to right modules.
We will denote by $A$-$\grMod$ the abelian category of graded left $A$-modules, where the morphisms are the $A$-linear maps preserving the grading. 
The space of morphisms in this category between two graded left $A$-modules $M$ and $M'$ will be denoted by $\hom_{A}(M,M')$. 
This category is provided with a \emph{shift functor} $(\place)[1]$ defined by $(M[1])_{n} = M_{n+1}$, where the underlying left $A$-module structure of $M[1]$ is the same as the one of $M$, 
and the action of the morphisms is trivial. 
We shall also denote $(\place)[d]$ the $d$-th iteration of the shift functor. 
The graded left $A$-module $M$ is said to be \emph{left bounded}, or also \emph{bounded below}, if there exists an integer $n_{0}$ such that $M_{n}=0$ for all $n<n_{0}$. 
Notice that the graded left $A$-modules which are left bounded form a full exact subcategory of $A$-$\grMod$.

Given graded left $A$-modules $N$ and $N'$, we recall the following notation:
\[     \mathcal{H}om_{A} (N,N') = \bigoplus\limits_{d\in \mathbb{Z}} \hom_{A} (N,N'[d]).     \]
We remark that, if $N$ is finitely generated, then $\mathcal{H}om_{A} (N,N') = \Hom_{A} (N,N')$, 
where the last morphism space is the usual one for left $A$-modules by forgetting the gradings (see \cite{NV}, Cor. 2.4.4).  

The following result is the graded version of the Nakayama Lemma.
\begin{lemma} 
\label{lemma:trivialmod}
Let $M$ be a left bounded graded left $A$-module. 
If $k \otimes_{A} M=0$, then $M$ is also trivial.
\end{lemma}
\begin{proof}
See \cite{B4}, Lemme 1.3 (see also \cite{C}, Exp. 15, Prop. 6). 
\end{proof}

The left $A$-module $M$ is said to be \emph{graded-free} if it is isomorphic to a direct sum of shifs $A[-l_{i}]$ of $A$. 
We remark that a bounded below graded left $A$-module $M$ is graded-free if and only if its underlying module (\textit{i.e.} forgetting the grading) is free, 
if and only if it is projective (as a graded module or not), if and only if $\Tor_{\bullet}^{A}(k,M) = 0$, for all $\bullet \geq 1$ (or just $\bullet = 1$). 
This will follow from the comments on projective covers. 

A surjective morphism $f : M \rightarrow M'$ in $A$-$\grMod$ is called \emph{essential} if for each morphism $g : N \rightarrow M$ in $A$-$\grMod$ such that $f\circ g$ is surjective, then $g$ is also surjective. 
As an application of the Nakayama Lemma we have the following result which characterizes essential surjective maps. 
\begin{lemma}
\label{lem:biyectiontensor}
Let $f : M \rightarrow M'$ be a morphism in the category of graded left (resp., right) $A$-modules.
Suppose that $M'$ is left bounded and that $f$ is surjective and essential. 
Then $\mathrm{id}_k \otimes_{A} f$ (resp., $f \otimes_{A} \mathrm{id}_{k}$) is bijective. 
Moreover, if $M$ is also left bounded, the converse holds.
\end{lemma}
\begin{proof}
See \cite{B4}, Lemme 1.5 (see also \cite{C}, Exp. 15, Prop. 7). 
\end{proof}

Let $M$ be a nontrivial object in $A$-$\grMod$. 
A \emph{projective cover} of $M$ is a pair $(P,f)$ such that $P\in$ $A$-$\grMod$ is projective and $f:P\rightarrow M$ is an essential surjective morphism.
We remark that every left bounded graded left $A$-module $M$ has a projective cover, which is unique up to (noncanonical) isomorphism (\textit{cf.} \cite{C}, Exp. 15, Thm. 2).
Moreover, given $M$ a bounded below left $A$-module, a projective cover may be explicitly constructed as follows. 
Since $M \neq 0$, the Nakayama lemma tells us that $M/(A_{>0} . M) \simeq k \otimes_{A} M$ is a nontrivial graded vector space. 
Consider a section $s$ of the canonical projection $M \rightarrow M/(A_{>0} . M) \simeq k \otimes_{A} M$. 
Now, we define $P = A \otimes (k \otimes_{A} M)$ together with the $A$-linear map $f : P \rightarrow M$ given by $f(a \otimes v) = a s(v)$, 
for $a \in A$ and $v \in k \otimes_{A} M$. 
Using the previous lemma one directly gets that $(P,f)$ is a projective cover of $M$. 

Recall that a (graded) projective resolution $(P_{\bullet},d_{\bullet})$ of a graded left $A$-module $M$ is \emph{minimal} if 
$d_{0} : P_{0} \rightarrow M$ is a projective cover (or equivalently, it is essential) and each of the maps $P_{i} \rightarrow \Ker(d_{i-1})$ induced  by $d_{i}$ is also essential, 
for all $i \in \NN$. 
We want to remark the important fact that, by iterating the process of considering projective covers for bounded below modules, 
one may easily prove that any bounded below graded left $A$-module has a minimal projective resolution (see \cite{B4}, Thm. 1.11).
If the left $A$-module $M$ has a minimal projective resolution $(P_{\bullet},d_{\bullet})$, for any other projective resolution $(Q_{\bullet},d'_{\bullet})$ of $M$, there exists an 
isomorphism of (augmented) complexes $Q_{\bullet} \simeq P_{\bullet} \oplus H_{\bullet}$, where $H_{\bullet}$ is acyclic (see \cite{B4}, Prop. 2.2). 
Additionally, the minimality assumption on the projective resolution implies that the differential of the induced complex $k \otimes_{A} P_{\bullet}$ vanishes (see \cite{C}, Exp. 15, Prop. 10, 
or \cite{B4}, Prop. 2.3), so if $(P_{\bullet},d_{\bullet})$ denotes such a minimal projective resolution, one also easily gets that $P_{\bullet} \simeq A \otimes \Tor_{\bullet}^{A} (k,M)$. 
Combining the results of the two previous sentences, it is trivial to see that if $(Q_{\bullet},d'_{\bullet})$ is a projective resolution of a graded left $A$-module $M$ 
having a minimal projective resolution, then the former is minimal if and only if the induced differential of $k \otimes_{A} Q_{\bullet}$ vanishes. 

If $N$ is left bounded, let $(P_{\bullet},d_{\bullet})$ be a (minimal) graded projective resolution of $N$. 
As usual, we denote
\[     \mathcal{E}xt_{A}^{i} (N,N') = H^{i} (\mathcal{H}om_{A} (P_{\bullet},N')),     \]
If the projective resolution of $N$ is composed of finitely generated projective left $A$-modules, there is a canonical identification 
$\mathcal{E}xt_{A}^{\bullet} (N,N') \simeq \Ext_{A}^{\bullet} (N,N')$. 
Moreover, using a very simple duality argument one can see that, if $M$ is a bounded below graded left $A$-module, then there is a canonical isomorphism of graded vector spaces 
\begin{equation}
\label{eq:isotorext}
     \mathcal{E}xt_{A}^{i}(M,k) \simeq \Tor^{A}_{i}(k,M)^{\#},
\end{equation}   
for all $i \in \NN_{0}$ (see \cite{B4}, Eq. (2.15), but \textit{cf.} also \cite{C}, Exp. 15, Prop. 2). 

We recall the beginning of the minimal projective resolution of the trivial left $A$-module $k$ for any nonnegatively graded connected algebra $A$. 
The analogous statements for the trivial right $A$-module $k$ are immediate. 
We know that it starts as 
\[
A \otimes V \overset{\delta_{1}}{\longrightarrow} A \overset{\delta_{0}}{\longrightarrow} k \longrightarrow 0,
\]
where $\delta_{0}$ is the augmentation of the algebra $A$, $V \simeq A_{>0}/(A_{>0} . A_{>0})$ is the vector space spanned by a minimal set of (homogeneous) generators of $A$ 
indicated before, and $\delta_{1}$ is the restriction of the product of $A$ (see \cite{C}, Exp. 15, end of Section 7). 
Furthermore, it is also well-known (and follows easily from the definition) that $\Ker(\delta_{1}) \simeq I/(I \otimes V)$ 
(as graded vector spaces), so there is an isomorphism of graded vector spaces $k \otimes_{A} \Ker(\delta_{1}) \simeq I/(T(V)_{>0} \otimes I + I \otimes T(V)_{>0})$, 
\textit{i.e.} we have an isomorphism of graded vector spaces $k \otimes_{A} \Ker(\delta_{1}) \simeq R$ 
(see \cite{Go}, Lemma 1, for complete expressions of the graded vector spaces $\Tor_{\bullet}^{A}(k,k)$, for $\bullet \in \NN_{0}$, in terms of $I$ and the irrelevant ideal $T(V)_{>0}$). 
Hence, $A \otimes R \rightarrow \Ker(\delta_{1})$ is a projective cover, and the beginning of the minimal projective resolution of the trivial $A$-module $k$ is of the form 
\[
A \otimes R \overset{\delta_{2}}{\longrightarrow} A \otimes V \overset{\delta_{1}}{\longrightarrow} A \overset{\delta_{0}}{\longrightarrow} k \longrightarrow 0,
\]
where $\delta_{2}$ is induced by the usual map $a \otimes v_{1} \dots v_{n} \mapsto a v_{1} \dots v_{n-1} \otimes v_{n}$ (see \cite{Le}, Lemme 1.2.11). 
Moreover, both left and right minimal projective resolutions of $k$ regarded as a left or right $A$-module respectively can be obtained from the minimal projective resolution 
$(P_{\bullet}^{b},\delta_{\bullet}^{b})$ of $A$ as an $A^{e}$-module easily. 
Indeed, the minimal projective resolution of the trivial left (resp., right) $A$-module $k$ is given by $P_{\bullet}^{b} \otimes_{A} k$ (resp., $k \otimes_{A} P_{\bullet}^{b}$). 
The fact that it is projective resolution of the trivial left (resp., right) $A$-module $k$ follows easily from the standard fact (see for instance \cite{BM}, Prop. 4.1), that, 
since the augmented complex $P_{\bullet}^{b} \rightarrow A$ is an exact and bounded below complex of projective right $A$-modules, it is homotopically trivial as a complex of right $A$-modules, 
so $P_{\bullet}^{b} \otimes_{A} k \rightarrow k$ is also homotopically trivial, and thus exact. 
To prove the minimality claim we use the fact we have already explained: a projective resolution $Q_{\bullet}$ of a bounded below graded left (resp., right) module $M$ over a locally finite dimensional nonnegatively graded algebra $B$ is minimal if and only if $k \otimes_{A} Q_{\bullet}$ (resp., $Q_{\bullet} \otimes_{A} k$) has vanishing differential. 
Now, by the minimality of $P_{\bullet}^{b}$, we get that $P_{\bullet}^{b} \otimes_{A^{e}} k = k \otimes_{A} P_{\bullet}^{b} \otimes_{A} k$ has vanishing differential, 
which further implies that $P_{\bullet} = P_{\bullet}^{b} \otimes_{A} k$ is minimal.   
Moreover, it is easy to see that the beginning the trivial $A$-bimodule $A$ is given by  
\begin{equation}
\label{eq:resprojminbi}
A \otimes R \otimes A \overset{\delta_{2}^{b}}{\longrightarrow} A \otimes V \otimes A \overset{\delta_{1}^{b}}{\longrightarrow} A \otimes A \overset{\delta_{0}^{b}}{\longrightarrow} A \longrightarrow 0,
\end{equation}
where $\delta_{0}^{b}$ is given by the multiplication of $A$, $\delta_{1}^{b}$ is determined by $\delta_{1}^{b}(a \otimes v \otimes a') = a v \otimes a' - a \otimes v a'$ 
and $\delta_{2}^{b}$ is the linear extension induced by $a \otimes v_{1} \dots v_{r} \otimes a' \mapsto \sum_{j=1}^{r} a v_{1} \dots v_{j-1} \otimes v_{j} \otimes v_{j+1} \dots v_{r} a'$.

\section{\texorpdfstring{Multi-Koszul algebras}{Multi-Koszul algebras}}
\label{sec:multikoszul}

\subsection{\texorpdfstring{Auxiliary results on subspaces of the tensor algebra}{Auxiliary results on subspaces of the tensor algebra}}
\label{subsec:inter}

We recall the following obvious property. 
\begin{fact}
\label{fact:inter}
Let $V, W$ be two graded vector subspaces of a graded vector space $U = \oplus_{n \in \ZZ} U_{n}$. 
Then, the intersection $V \cap W$ is also a graded vector subspace of $U$ such that $(V \cap W)_{n} = V_{n} \cap W_{n}$, for $n \in \ZZ$. 
\end{fact}

Let $\mu$ denote the multiplication of the tensor algebra $T(V)$, and, for $n \geq 2$, $\mu^{(n)} : T(V)^{\otimes n} \rightarrow T(V)$ the $(n-1)$-th iteration of $\mu$, 
\textit{i.e.} it is defined by the recursive process given by $\mu^{(2)} = \mu$ and $\mu^{(i+1)} = \mu \circ (\mathrm{id}_{T(V)} \otimes \mu^{(i)})$, for $i \in \NN_{\geq 2}$. 
Note that $\mu^{(n)}$ is surjective (if $V$ is nontrivial) but not injective, for all $n \in \NN_{\geq 2}$. 
Given $W_{1}, \dots, W_{m}$ vector subspaces of $T(V)$, we shall denote by $W_{1} \dots W_{m}$ the vector subspace of $T(V)$ 
given by the image of the vector subspace $W_{1} \otimes \dots \otimes W_{m} \subseteq T(V)^{\otimes m}$ under the map $\mu^{(m)}$. 
Since the product of the tensor algebra is denoted by the $\otimes$ symbol, it is usual to denote this image also by $W_{1} \otimes \dots \otimes W_{m}$, 
even though this product differs in principle from the usual tensor product of the vector spaces $W_{1}, \dots, W_{m}$ 
(\textit{e.g.} take $W_{1} = \mathrm{span}_{k}\cl{x,x^{2}}$, and $W_{2} = \mathrm{span}_{k}\cl{x^{2},x^{3}}$. 
Then, the vector subspace $W_{1} . W_{2}$ of $T(V)$ has basis $\{ x^{3}, x^{4}, x^{5}\}$, whereas the usual tensor product $W_{1} \otimes W_{2}$ has dimension $4$).  
We shall also adopt this convention when we are talking about vector subspaces of the tensor algebra, unless we explicitly state otherwise.
This notation is coherent in the sense that, given $W_{1}, \dots, W_{m}$ vector subspaces of $T(V)$, $W_{1} \otimes \dots \otimes W_{m}$ denotes the vector subspace of $T(V)$ 
that consists of the elements of $T(V)$ given by the sums of terms of the form $w_{1} \otimes \dots \otimes w_{m}$, where $w_{i} \in W_{i}$, 
and the $\otimes$ symbol indicates the product of the tensor algebra $T(V)$. 

A word of caution should be stated here. 
Given $W_{1}, \dots, W_{m}$ vector subspaces of $T(V)$ as before, and linear maps $f_{i} : W_{i} \rightarrow W'_{i}$, for $i = 1, \dots, m$, 
where $W'_{i}$ are arbitrary vector spaces, if we use the previous notation then there may be no canonical manner to define a map 
$f_{1} \otimes f_{m} : W_{1} \otimes \dots \otimes W_{m} \rightarrow W'_{1} \otimes \dots \otimes W'_{m}$, 
since an element of $W_{1} \otimes \dots \otimes W_{m}$ could be written inside $T(V)$ in inequivalent manners (see however Corollary \ref{corollary:maps}). 
Note that, since we are not assuming that the vector spaces $W'_{1}, \dots, W'_{m}$ are vector subspaces of a tensor algebra, $W'_{1} \otimes \dots \otimes W'_{m}$ denotes the usual tensor product.    

Let $U \subseteq T(V)$ be a vector subspace of the tensor algebra on a vector space $V$. 
We say that $U$ is \emph{left tensor-intersection faithful} (resp., \emph{right tensor-intersection faithful}) if $(U \cap W_{1}) \cap (U \otimes W_{2}) = U \otimes (W_{1} \cap W_{2})$ (resp., $(W_{1} \otimes U) \cap (W_{2} \otimes U) = (W_{1} \cap W_{2}) \otimes U$), for all vector subspaces $W_{1}, W_{2} \subseteq T(V)$. 
Moreover, $U$ is said to be \emph{tensor-intersection faithful} if it is left and right tensor-intersection faithful. 

We have the following simple characterization of tensor-intersection faithfulness, that we will use extensively in the sequel. 
We suspect that it is well-known among the experts but we were unable to find any reference whatsoever. 
\begin{proposition}
\label{proposition:imp}
Let $U \subseteq T(V)$ be a vector subspace of the tensor algebra. 
Then, the following conditions are equivalent:
\begin{itemize} 
\item[(i)] $U$ is left (resp., right) tensor-intersection faithful. 

\item[(ii)] $U \cap (U \otimes T(V)_{>0}) = 0$ (resp., $U \cap (T(V)_{>0} \otimes U) = 0$). 

\item[(iii)] Given $I$ a finite set of indices, a linearly independent family $\{ u_{i} : i \in I \} \subseteq U$, and arbitrary elements $w_{i} \in T(V)$, for $i \in I$, 
if $\sum_{i \in I} u_{i} \otimes w_{i}$ (resp., $\sum_{i \in I} w_{i} \otimes u_{i}$) vanishes, then $w_{i} = 0$, for all $i \in I$. 
\end{itemize}
\end{proposition}
\begin{proof}
The implication $(i) \Rightarrow (ii)$ is clear. 
Indeed, it easily follows by considering $W_{1} = k$ and $W_{2} = T(V)_{> 0}$. 

We will now prove the implication $(ii) \Rightarrow (iii)$. 
We shall only consider the case corresponding to the assumption $U \cap (U \otimes T(V)_{>0}) = 0$, because the other case is analogous. 
Let $\{ v_{s} \}_{s \in S}$ be a basis of $V$. 
Given $n \in \NN_{0}$, we will denote by $\bar{s} = (s_{1}, \dots, s_{n})$ a typical element of $S^{n}$. 
Let us write 
\[     w_{i} = \sum_{n \in \NN_{0}} \sum_{\bar{s} \in S^{n}} c_{i,n,\bar{s}} \, v_{s_{1}} \dots v_{s_{n}},     \]
where $c_{i,n,\bar{s}} \in k$, such that the sum is of finite support, \textit{i.e.} fixed $i \in I$, the coefficients $c_{i,n,\bar{s}}$ are almost all zero. 
We have to prove that they are in fact all zero. 
Let us suppose that this is not the case, and consider $n_{0}$ the first nonnegative integer satisfying that there exists $i_{0} \in I$ and $\bar{s}_{0} \in S^{n_{0}}$ such that 
$c_{i_{0},n_{0},\bar{s}_{0}} \neq 0$. 
We will write $\bar{s}_{0} = (s_{0,1}, \dots, s_{0,n_{0}})$. 
On the one hand, we write 
\begin{align*}
     \sum_{i \in I} u_{i} \otimes w_{i} &= \sum_{i \in I} \sum_{n \in \NN_{0}} \sum_{\bar{s} \in S^{n}} c_{i,n,\bar{s}} \, u_{i} \otimes v_{s_{1}} \dots v_{s_{n}} 
     \\
     &= \sum_{n \in \NN_{0}} \sum_{\bar{s} \in S^{n}} \sum_{i \in I} c_{i,n,\bar{s}} \, u_{i} \otimes v_{s_{1}} \dots v_{s_{n}} 
     \\
     &= \sum_{n \geq n_{0}} \sum_{\bar{s} \in S^{n}} \sum_{i \in I} c_{i,n,\bar{s}} \, u_{i} \otimes v_{s_{1}} \dots v_{s_{n}}.     
\end{align*}
Since this sum vanishes, by the definition of the tensor algebra we also have that the sum 
\[    \sum_{n \in \NN_{0}} \sum_{\bar{s} \in S^{n}} \sum_{i \in I} c_{i,n+n_{0},(\bar{s},\bar{s}_{0})} \, u_{i} \otimes v_{s_{1}} \dots v_{s_{n}} v_{s_{0,n_{0}}} \dots v_{s_{0,n_{0}}}     \]
vanishes, where $(\bar{s},\bar{s}_{0})$ denotes the $(n+n_{0})$-tuple $(s_{1},\dots,s_{n},s_{0,1}, \dots, s_{0,n_{0}})$. 
Thus, 
\[     \sum_{n \in \NN_{0}} \sum_{\bar{s} \in S^{n}} \sum_{i \in I} c_{i,n+n_{0},(\bar{s},\bar{s}_{0})} \, u_{i} \otimes v_{s_{1}} \dots v_{s_{n}} = 0.     \]
Define 
\[     \omega_{>0} = \sum_{n \in \NN} \sum_{\bar{s} \in S^{n}} \sum_{i \in I} c_{i,n+n_{0},(\bar{s},\bar{s}_{0})} \, u_{i} \otimes v_{s_{1}} \dots v_{s_{n}},     \]
so
\[     \sum_{n \in \NN_{0}} \sum_{\bar{s} \in S^{n}} \sum_{i \in I} c_{i,n+n_{0},(\bar{s},\bar{s}_{0})} \, u_{i} \otimes v_{s_{1}} \dots v_{s_{n}} 
                                                                                                            = \big(\sum_{i \in I} c_{i,n_{0},\bar{s}_{0}} \, u_{i}\big) + \omega_{>0} = 0.     \]
If $\omega_{>0}$ vanishes, then the sum $\sum_{i \in I} c_{i,n_{0},\bar{s}_{0}} \, u_{i} \in U$ should also vanish. 
On the other hand, if $\omega_{>0}$ does not vanish, the sum $\sum_{i \in I} c_{i,n_{0},\bar{s}_{0}} \, u_{i} \in U$ should be in $U \otimes T(V)_{>0}$, 
so, by the assumption on $U$, this sum should again be trivial. 
In any case we have that $\sum_{i \in I} c_{i,n_{0},\bar{s}_{0}} \, u_{i} \in U$ vanishes. 
Now, since the $\{ u_{i} \}_{i \in I}$ form a linearly independent set, we get that $c_{i,n_{0},\bar{s}_{0}} = 0$, for all $i \in I$. 
This is a contradiction, by the assumption that $c_{i_{0},n_{0},\bar{s}_{0}} \neq 0$. 
Hence, all the coefficients $c_{i,n,\bar{s}}$ vanish, which proves that the corresponding statement of item $(iii)$ holds. 

Let us now prove the implication $(iii) \Rightarrow (i)$. 
As before, we shall only prove the implication for the left side conditions, for the other case is analogous.  
We assume thus that $U$ satisfies that, given $I$ a finite set of indices, a linearly independent family $\{ u_{i} : i \in I \} \subseteq U$, and arbitrary elements $w_{i} \in T(V)$, for $i \in I$, 
if $\sum_{i \in I} u_{i} \otimes w_{i}$ vanishes, then $w_{i} = 0$, for all $i \in I$. 
We will prove that in this case $(U \otimes W) \cap (U \otimes W') = U \otimes (W \cap W')$. 
Note that $(U \otimes W) \cap (U \otimes W') \supseteq U \otimes (W \cap W')$ is always true by obvious reasons. 
We shall prove the other contention. 
Consider a nontrivial element $\omega \in (U \otimes W) \cap (U \otimes W')$, and fix a basis $\{ u_{s} \}_{s \in S}$ of $U$. 
Then, there exists two finite subsets $S_{1}, S_{2} \subseteq S$, and elements $w_{s} \in W$ for $s \in S_{1}$ and $w'_{s} \in W'$ for $s \in S_{2}$, such that 
$\omega = \sum_{s \in S_{1}} u_{s} \otimes w_{s} = \sum_{s \in S_{2}} u_{s} \otimes w'_{s}$. 
Setting $w_{s} = 0$ if $s \in S \setminus S_{1}$, and $w'_{s} = 0$ if $s \in S \setminus S_{2}$, we may write 
$\omega = \sum_{s \in S} u_{s} \otimes w_{s} = \sum_{s \in S} u_{s} \otimes w'_{s}$, so $\sum_{s \in S} u_{s} \otimes (w_{s} - w'_{s}) = 0$. 
The assumption implies that $w_{s} = w'_{s}$, for all $s \in S$, so $w_{s} \in W \cap W'$, for all $s \in S$. 
Hence, $\omega \in U \otimes (W \cap W')$, which proves the assertion. 
The proposition is thus proved.  
\end{proof}

\begin{example}
\label{example:tif}
The previous proposition immediately implies that $V$ is a tensor-intersection faithful vector subspace of the tensor algebra $T(V)$. 
Furthermore, it also tells us that a graded vector subspace $R \subseteq T(V)$ of the tensor algebra on a positively graded vector space $V$ satisfying condition \eqref{eq:sprel} (\textit{e.g.} a space of relations $R$ of an algebra $T(V)/I$) is tensor-intersection faithful. 
\end{example}

\begin{remark}
\label{remark:tensbuenomalo}
Note that if $U \subseteq T(V)$ satisfies that $U \cap (U \otimes T(V)_{>0}) = 0$ (resp., $U \cap (T(V)_{>0} \otimes U) = 0$), then any vector subspace $U'$ of $U$ 
trivially satisfies the same condition.
The proposition tells us thus that any vector subspace of a left (resp., right) tensor-intersection faithful space is also left (resp., right) tensor-intersection faithful. 

Moreover, it is clear that the tensor product $U_{1} \otimes U_{2}$ of the two left (resp., right) tensor-intersection faithful vector subspaces $U_{i} \subseteq T(V)$, for $i=1,2$, satisfies the same condition.
\end{remark}

In exactly the same manner as in the proof of the implication $(iii) \Rightarrow (i)$ of Proposition \ref{proposition:imp} we may argue to prove the following result. 
\begin{corollary}
\label{corollary:interincl2}
Let $V$ be a vector space, and $T(V)$ the tensor algebra. 
Let us consider two vector subspaces $U' \subseteq U \subseteq T(V)$ such that  $U$ is left (resp., right) tensor-intersection faithful.
Then, given any two vector subspaces $W' \subseteq W \subseteq T(V)$, we have that 
\[     (U' \otimes W) \cap (U \otimes W') = U' \otimes W' \hskip 0.1cm \text{(resp.,  $(W \otimes U') \cap (W' \otimes U) = W' \otimes U'$)},     \]
as vector subspaces of $T(V)$. 
\end{corollary}
\begin{proof}
We shall only prove the implication for the left side conditions, for the other case is analogous.  
Since $U$ is left tensor-intersection faithful, the previous proposition tells us that, given $I$ a finite set of indices, a linearly independent family $\{ u_{i} : i \in I \} \subseteq U$, and arbitrary elements $w_{i} \in T(V)$, for $i \in I$, if $\sum_{i \in I} u_{i} \otimes w_{i}$ vanishes, then $w_{i} = 0$, for all $i \in I$. 
We recall that, by the previous Remark, $U'$ satisfies the same property. 
We will prove that in this case $(U' \otimes W) \cap (U \otimes W') = U' \otimes W'$. 

Note that $(U' \otimes W) \cap (U \otimes W') \supseteq U' \otimes W'$ is always true by obvious reasons. 
We shall prove the other contention. 
Consider a nontrivial element $\omega \in (U' \otimes W) \cap (U \otimes W')$, and fix a basis $\{ u_{s} \}_{s \in S}$ of $U$. 
We furthermore assume that there exists a subset $S' \subseteq S$ such that $\{ u_{s'} \}_{s' \in S'}$ is a basis of $U'$. 
Then, there exists two families of almost all zero elements $w_{s} \in W'$ for $s \in S$ and $w'_{s'} \in W$ for $s' \in S'$, such that 
$\omega = \sum_{s \in S} u_{s} \otimes w_{s} = \sum_{s' \in S'} u_{s'} \otimes w'_{s'}$. 
Setting $w'_{s} = 0$ if $s \in S \setminus S'$, we may write 
$\omega = \sum_{s \in S} u_{s} \otimes w_{s} = \sum_{s \in S} u_{s} \otimes w'_{s}$, so $\sum_{s \in S} u_{s} \otimes (w_{s} - w'_{s}) = 0$. 
The assumption now implies that $w_{s} = w'_{s}$, for all $s \in S$. 
Hence, $\omega \in U' \otimes W'$, which proves the assertion. 
\end{proof}

\begin{remark}
We strongly remark that in the statements of the previous proposition and corollary, all the intersections of spaces considered are taken by regarding all the corresponding spaces 
as vector subspaces of the tensor algebra $T(V)$. 
For instance, in Proposition \ref{proposition:imp}, when considering the intersection of $W$ and $W'$, and the intersection of $U \otimes W$ and $U \otimes W'$, 
we are regarding $W$, $W'$, $U \otimes W$ and $U \otimes W'$ as vector subspaces of $T(V)$.  
The same comments apply to Corollary \ref{corollary:interincl2}. 
\end{remark}

We have the following interesting corollaries of the previous proposition, the first one of these being a direct consequence of the implication $(ii) \Rightarrow (iii)$. 
\begin{corollary}
\label{corollary:interind0}
Let $V$ be a vector space, and $T(V)$ the tensor algebra. 
Let us consider a left (resp., right) tensor-intersection faithful vector subspaces $U \subseteq T(V)$.
If $\{ U_{i} \}_{i \in I}$ is an arbitrary family of independent vector subspaces of the $U$, then 
the family of vector subspaces of the tensor algebra given by $\{ U_{i} \otimes T(V) \}_{i \in I}$ (resp., $\{ T(V) \otimes U_{i} \}_{i \in I}$) is also independent. 
\end{corollary}

\begin{corollary}
\label{corollary:interind}
Let $V$ be a vector space, and $T(V)$ the tensor algebra. 
Let us consider a left (resp., right) tensor-intersection faithful vector subspaces $U \subseteq T(V)$.
If $\{ W_{i} \}_{i \in I}$ is an arbitrary family of independent vector subspaces of the $T(V)$, then 
the family of vector subspaces of the tensor algebra given by $\{ U \otimes W_{i} \}_{i \in I}$ (resp., $\{ W_{i} \otimes U \}_{i \in I}$) is also independent. 
\end{corollary}
\begin{proof}
We shall prove the statement for the left tensor-intersection faithfulness assumption on $U$, for the right case is completely parallel. 
Let $w_{i} \in W_{i}$ and $u_{i} \in U$ be collections of almost all zero elements such that $\sum_{i \in I} u_{i} \otimes w_{i}$ vanishes. 
We have to prove that each term of the sum vanishes. 
Without loss of generality we may assume that $I$ is in fact finite, and that $w_{i}$ is nonvanishing, for all $i \in I$. 
Let $\{ \bar{u}_{s} \}_{s \in S}$ be a basis of $U$, and write $u_{i} = \sum_{s \in S} c_{s i} \bar{u}_{s}$, where $c_{s i} \in k$. 
We then have that 
\[     \sum_{i \in I} u_{i} \otimes w_{i} = \sum_{i \in I} \sum_{s \in S} c_{s i} \bar{u}_{s} \otimes w_{i} = \sum_{s \in S} \bar{u}_{s} \otimes \big(\sum_{i \in I} c_{s i}  w_{i}\big).     \]
Since $U$ is left tensor-intersection faithful, we have that $\sum_{i \in I} c_{s i}  w_{i}$ vanishes for all $s \in S$. 
Taking into account that the family $\{ W_{i} \}_{i \in I}$ is independent we get that $c_{s i}  w_{i} = 0$, for all $s \in S$ and $i \in I$, 
which yields that $c_{s i} = 0$, for all $s \in S$ and $i \in I$, for the elements $w_{i}$ are nonzero for all $i \in I$. 
So, $u_{i} = 0$, for all $i \in I$, which in turn implies that $u_{i} \otimes w_{i}$ vanishes for all $i \in I$, which was to be shown.  
The corollary is thus proved.
\end{proof}

Finally, we state the last consequence of Proposition \ref{proposition:imp}.
\begin{corollary}
\label{corollary:maps}
Let $V$ be a vector space, $T(V)$ the tensor algebra, $U \subseteq T(V)$ a left (resp., right) tensor-intersection faithful vector subspace, and $W \subseteq T(V)$. 
Given arbitrary vector spaces $U'$ and $W'$, and linear maps $f : U \rightarrow U'$ and $g : W \rightarrow W'$, then the linear map $f \otimes g : U \otimes W \rightarrow U' \otimes W'$ 
(resp., $g \otimes f : W \otimes U \rightarrow W' \otimes U'$) given by the usual expression $(f \otimes g)(\sum_{i} u_{i} \otimes w_{i}) = \sum_{i} f(u_{i}) \otimes g(w_{i})$ 
(resp., $(g \otimes f)(\sum_{i} w_{i} \otimes u_{i}) = \sum_{i} g(w_{i}) \otimes f(u_{i})$), where $u_{i} \in U$, $w_{i} \in W$, and $I$ is a finite set of indices, 
is well-defined, where as usual $U \otimes W$ (resp., $W \otimes U$) denotes the product of $U$ and $W$ inside $T(V)$, 
whereas $U' \otimes W'$ (resp., $W' \otimes U'$) denotes the usual tensor product. 
\end{corollary}
\begin{proof}
We shall prove the corollary under the left tensor-intersection faithfulness assumption on $U$, because the right case is analogous. 
Let $\omega \in U \otimes W$ be given in two different manners by the sum $\sum_{i \in I} u'_{i} \otimes w'_{i}$, where $u'_{i} \in U$ and $w'_{i} \in W$, and $I$ is a finite set of indices, 
and by the sum $\sum_{j \in J} u''_{j} \otimes w''_{j}$, where $u''_{j} \in U$ and $w''_{j} \in W$, and $J$ is another finite set of indices. 
We have to prove that $\sum_{i \in I} f(u'_{i}) \otimes g(w'_{i}) = \sum_{j \in J} f(u''_{j}) \otimes g(w''_{j})$, as elements of the usual tensor product $U' \otimes W'$. 
Fix a basis $\{ u_{s} \}_{s \in S}$ of $U$, and write $u'_{i} = \sum_{s \in S} c'_{s i} u_{s}$ and $u''_{j} = \sum_{s \in S} c''_{s j} u_{s}$, for all $i \in I$ and $j \in J$, respectively, 
where $c'_{s i}, c''_{s j} \in k$. 
Then, 
\[     \omega = \sum_{i \in I} u'_{i} \otimes w'_{i} = \sum_{s \in S} u_{s} \otimes \Big(\sum_{i \in I} c'_{s i} w'_{i}\Big),     \]
and 
\[     \omega = \sum_{j \in J} u''_{j} \otimes w''_{j} = \sum_{s \in S} u_{s} \otimes \Big(\sum_{j \in J} c''_{s j} w''_{j}\Big).     \]
Since $U$ is a left tensor-intersection faithful vector subspace of the tensor algebra, we have that 
\[     \sum_{i \in I} c'_{s i} w'_{i} = \sum_{j \in J} c''_{s j} w''_{j},     \]
for all $s \in S$. 
Let us denote this element $w_{s}$. 

In particular, we have that 
\begin{align*}
   \sum_{i \in I} f(u'_{i}) \otimes g(w'_{i}) &= \sum_{i \in I} \sum_{s \in S} c'_{s i} f(u_{s}) \otimes g(w'_{i}) = \sum_{s \in S} f(u_{s}) \otimes \Big(\sum_{i \in I} c'_{s i} g(w'_{i})\Big) 
   \\
        &= \sum_{s \in S} f(u_{s}) \otimes g(w_{s}) = \sum_{s \in S} f(u_{s}) \otimes \Big(\sum_{j \in J} c''_{s j} g(w''_{j})\Big) 
   \\
    &= \sum_{j \in J} \sum_{s \in S} c''_{s j} f(u_{s}) \otimes g(w''_{j}) = \sum_{j \in J} f(u''_{j}) \otimes g(w''_{j}),
\end{align*}
which was to be shown. 
\end{proof}

\subsection{\texorpdfstring{The definition of multi-Koszul algebras}{The definition of multi-Koszul algebras}}
\label{subsec:defimultikoszul}

Given $i \in \NN_{0}$ let us consider the family $\{\cap_{j=0}^{N} V^{(j)} \otimes R^{(i)} \otimes V^{(N-j)}\}_{N \in \NN_{0}}$ of graded vector subspaces of $T(V)$. 
If $i = 0$, the previous family coincides with $\{ V^{(N)} \}_{N \in \NN_{0}}$, which is trivially seen to be independent. 
Moreover, if $i \in \NN$, we claim that this family is also independent. 
Indeed, this follows from the fact that the family $\{ V^{(N)} \otimes R^{(i)} \}_{N \in \NN_{0}}$ of graded vector subspaces of $T(V)$ is independent by Corollary \ref{corollary:interind}, 
and the $N$-th member of the former family is a vector subspace of the $N$-th member of the latter family. 
Note furthermore that any family of vector subspaces of the previous ones is thus independent. 

Set the graded vector subspace of $T(V)$ given by the direct sum 
\begin{equation}
\label{eq:iTN}
     {}^{i}T = \bigoplus_{N \in \NN_{0}} \big(\bigcap_{j=0}^{N} V^{(j)} \otimes R^{(i)} \otimes V^{(N-j)}\big).     
\end{equation}
We may thus consider ${}^{i}T = \oplus_{N \in \NN_{0}} {}^{i}T^{N}$, where ${}^{i}T^{N} = \cap_{j=0}^{N} V^{(j)} \otimes R^{(i)} \otimes V^{(N-j)}$, 
provided with another grading given by the index $N$ of the previous direct sum. 
Since each ${}^{i}T^{N}$ is a graded vector space for the grading coming form the grading on $V$, ${}^{i}T$ is in fact a bigraded vector space. 
We shall refer the grading of ${}^{i}T$ coming from the grading on $V$ as \emph{usual} if more indications are necessary, 
even though we shall call it in general the grading of ${}^{i}T$, without further specifications.  
On the other hand, we will refer to the new grading given by \eqref{eq:iTN} as \emph{special}. 
Note that ${}^{0}T = T(V)$, and that the usual and special gradings coincide in this case. 

Given $N \in \NN$, we recall that a \emph{partition} of $N$ is an $N$-tuple of the form $\bar{n} = (n_{1},\dots,n_{N}) \in \NN_{0}^{N}$ such that $n_{1} + \dots + n_{N} = N$. 
The \emph{length} of the partition $\bar{n}$ is defined as the greatest positive integer $l \leq N$ such that $n_{l} \neq 0$, and it is denoted by $l(\bar{n})$. 
We will denote by $\mathrm{Par}(N)$ the set of partitions of $N$, and by $\mathrm{Par}_{l}(N)$ the set of partitions of length $l$. 
It is clear that $\mathrm{Par}(N) = \cup_{l=1}^{N} \mathrm{Par}_{l}(N)$, and the same occurs for the sets of decreasing partitions.
If $N=0$, we set $\mathrm{Par}(0) = \emptyset$. 

Finally, we will define the family $\{ J_{i} \}_{i \in \NN_{0}}$ of graded vector subspaces of $T(V)$ recursively as follows. 
We set $J_{0} = k$. 
We shall first define the spaces indexed by even integers. 
Suppose we have defined $J_{0}, J_{2}, \dots J_{2i}$, for some $i \in \NN_{0}$. 
Then, we set
\begin{equation}
\label{eq:j_impar}
     J_{2(i + 1)} = \Big(\sum_{N \in \NN} \big(\bigcap_{\begin{scriptsize}\begin{matrix}\bar{n} \in \mathrm{Par}(i) \\ \bar{m} \in \mathrm{Par}_{i+1}(N) \end{matrix}\end{scriptsize}} 
V^{(m_{1})} \otimes J_{2 n_{1}} \otimes \dots \otimes V^{(m_{i})} \otimes J_{2 n_{i}} \otimes V^{(m_{i+1})}\big)\Big) \cap R^{(i+1)}.     
\end{equation}
On the other hand, we define 
\begin{equation}
\label{eq:j_par}
     J_{2i + 1} = (V \otimes J_{2i}) \cap (J_{2i} \otimes V),     
\end{equation}

Notice that $J_{1} = V$, $J_{2} = R$ and $J_{3} = (V \otimes R) \cap (R \otimes V)$ as graded vector subspaces of $T(V)$. 
On the other hand, note that by construction we have that $J_{2 j} \subseteq R^{(j)}$, for all $j \in \NN_{0}$, and thus $J_{2 j + 1} \subseteq (V \otimes R^{(j)}) \cap (R^{(j)} \otimes V)$, 
also for $j \in \NN_{0}$. 

Furthermore, we claim that the sum appearing in the definition \eqref{eq:j_impar} of $J_{2i+2}$ is in fact direct.  
In order to prove so, note that 
\[     \bigcap_{\begin{scriptsize}\begin{matrix}\bar{n} \in \mathrm{Par}(i) \\ \bar{m} \in \mathrm{Par}_{i+1}(N) \end{matrix}\end{scriptsize}} 
V^{(m_{1})} \otimes J_{2 n_{1}} \otimes \dots \otimes V^{(m_{i})} \otimes J_{2 n_{i}} \otimes V^{(m_{i+1})}    \]
is trivially included in 
\[     \bigcap_{l=0}^{N} V^{(j)} \otimes J_{2 i} \otimes V^{(N-l)},     \]
as one sees just by considering partitions $\bar{n} \in \mathrm{Par}_{1}(i)$. 
It suffices to prove that the family given by these graded vector subspaces of $T(V)$, for $N \in \NN$, is independent. 
But, since $J_{2 i} \subseteq R^{(i)}$, we see that $\bigcap_{l=0}^{N} V^{(l)} \otimes J_{2i} \otimes V^{(N-l)} \subseteq {}^{i}T^{N}$, for all $N \in \NN$, 
we conclude that the sum appearing in \eqref{eq:j_impar} is in fact direct, so
\begin{equation}
\label{eq:j_imparbis}
     J_{2(i + 1)} = \Big(\bigoplus_{N \in \NN} \big(\bigcap_{\begin{scriptsize}\begin{matrix}\bar{n} \in \mathrm{Par}(i) \\ \bar{m} \in \mathrm{Par}_{i+1}(N) \end{matrix}\end{scriptsize}} 
V^{(m_{1})} \otimes J_{2 n_{1}} \otimes \dots \otimes V^{(m_{i})} \otimes J_{2 n_{i}} \otimes V^{(m_{i+1})}\big)\Big) \cap R^{(i+1)}.  
\end{equation} 
Note that the direct summand appearing in the sum of the first term of the intersection of the right member and corresponding to $N = 1$ is $(V \otimes J_{2 i}) \cap (J_{2 i} \otimes V)$, 
which coincides with $J_{2 i + 1}$, by \eqref{eq:j_par}.  

Moreover, we claim that 
\begin{equation}
\label{eq:j_imparbisbis}
     J_{2(i + 1)} = \Big(\bigoplus_{N \in \NN_{\geq 2}} \big(\bigcap_{\begin{scriptsize}\begin{matrix}\bar{n} \in \mathrm{Par}(i) \\ \bar{m} \in \mathrm{Par}_{i+1}(N) \end{matrix}\end{scriptsize}} 
V^{(m_{1})} \otimes J_{2 n_{1}} \otimes \dots \otimes V^{(m_{i})} \otimes J_{2 n_{i}} \otimes V^{(m_{i+1})}\big)\Big) \cap R^{(i+1)}.  
\end{equation}
The inclusion of the right member inside the left one is clear, so we only need to prove the other contention. 
Let $\omega \in J_{2(i + 1)}$. 
By \eqref{eq:j_imparbis}, we have that $\omega = \omega_{1} + \omega_{2}$, where $\omega_{1} \in J_{2i+1}$, 
and $\omega_{2}$ belongs to the first term of the intersection of the right member of \eqref{eq:j_imparbisbis}. 
We have to prove that $\omega_{1}$ vanishes. 
Since $J_{2i+1} \subseteq V \otimes R^{(i)}$, we get that $\omega_{1} \in V \otimes R^{(i)}$. 
Analogously, since 
\[     \bigcap_{\begin{scriptsize}\begin{matrix}\bar{n} \in \mathrm{Par}(i) \\ \bar{m} \in \mathrm{Par}_{i+1}(N) \end{matrix}\end{scriptsize}} 
V^{(m_{1})} \otimes J_{2 n_{1}} \otimes \dots \otimes V^{(m_{i})} \otimes J_{2 n_{i}} \otimes V^{(m_{i+1})}    \]
is trivially included in 
\[     \bigcap_{l=0}^{N} V^{(j)} \otimes J_{2 i} \otimes V^{(N-l)} \subseteq V^{(N)} \otimes R^{(i)},     \]
we have that $\omega_{2} \in T(V)^{\geq 2} \otimes R^{(i)}$. 
We also have that $\omega \in R^{(i+1)} \subseteq R \otimes R^{(i)} \subseteq T(V)^{\geq 2} \otimes R^{(i)}$. 
Hence, $\omega_{1} = \omega - \omega_{2}$ should be an element of the intersection of $V \otimes R^{(i)}$ and 
$T(V)^{\geq 2} \otimes R^{(i)}$, which is trivial by Proposition \ref{proposition:imp}, so $\omega_{1} = 0$, which was to be shown.  

The inclusions $J_{2 j} \subseteq R^{(j)}$, and $J_{2 j + 1} \subseteq (V \otimes R^{(j)}) \cap (R^{(j)} \otimes V)$, together with Example \ref{example:tif} and 
Remark \ref{remark:tensbuenomalo}, imply the following useful result. 
\begin{lemma}
\label{lemma:interj}
Let $\{ J_{i} \}_{i \in \NN_{0}}$ be the collection of graded vector subspaces of $T(V)$ defined by $J_{0} = k$ and the recursive identities \eqref{eq:j_impar} and \eqref{eq:j_par}. 
For all $i \in \NN_{0}$, $J_{i}$ is a tensor-intersection faithful vector subspace of $T(V)$.   
\end{lemma}

Given $j \in \NN_{0}$, let us consider the family of graded vector subspaces $\{ \cap_{l=0}^{N} V^{(l)} \otimes J_{2 j + 1} \otimes V^{(N-l)}\}_{N \in \NN_{0}}$ of the tensor algebra $T(V)$. 
Since the $N$-th member of this family is trivially included $R^{(j)} \otimes V^{(N+1)}$, and the family $\{ R^{(j)} \otimes V^{(N+1)} \}_{N \in \NN_{0}}$ is independent by Corollary \ref{corollary:interind}, the former family is independent. 
We define thus 
\begin{equation}
\label{eq:jtildeT}
     {}^{j}\tilde{T} = \bigoplus_{N \in \NN_{0}} \big(\bigcap_{l=0}^{N} V^{(l)} \otimes J_{2 j + 1} \otimes V^{(N-l)}\big),     
\end{equation}
which is considered as a bigraded vector space, where the first grading, which will be called \emph{usual}, comes from that of $V$, and the second one, called \emph{special}, 
comes from the direct sum decomposition of the definition. 
Note that ${}^{0}\tilde{T} = S(T(V)_{>0})$ and ${}^{1}\tilde{T} = S({}^{1}T^{>0})$ as bigraded vector spaces, where $S$ is the \emph{shift} functor for the special grading, 
\textit{i.e.} $S(E)^{N} = E^{N+1}$. 
Furthermore, by \eqref{eq:j_imparbis} we have that $J_{2 j} \subseteq {}^{j-1}\tilde{T}$, for all $j \in \NN$.  

We have the following simple result. 
\begin{lemma}
\label{lemma:2recur}
Let $\{ J_{i} \}_{i \in \NN_{0}}$ be the collection of graded vector subspaces of $T(V)$ defined by $J_{0} = k$ and the recursive identities \eqref{eq:j_impar} and \eqref{eq:j_par}. 
For all $i \in \NN_{\geq 2}$, we have the inclusion 
\[     J_{i} \subseteq (R \otimes J_{i-2}) \cap (J_{i-2} \otimes R).     \]   
\end{lemma}
\begin{proof}
The inclusion of the statement is equivalent to the two inclusions
\[     J_{i} \subseteq R \otimes J_{i-2}, \hskip 1cm \text{and} \hskip 1cm J_{i} \subseteq J_{i-2} \otimes R.     \]   
We shall only prove the second one, for the other is completely analogous. 
Let us first suppose that $i$ is even. 
In this case, we trivially see from definition \eqref{eq:j_impar} that 
\[     J_{i} \subseteq (J_{i-2} \otimes T(V)_{>0}) \cap R^{(\frac{i}{2})},      \]
by only considering the partitions of length $1$. 
On the other hand, since $J_{i-2} \subseteq R^{((i - 2)/2)}$, Corollary \ref{corollary:interincl2} and Lemma \ref{lemma:interj} yield that 
\[     J_{i} \subseteq (J_{i-2} \otimes T(V)_{>0}) \cap (R^{(\frac{i-2}{2})} \otimes R) \subseteq J_{i-2} \otimes R.      \]

If $i$ is odd, from \eqref{eq:j_impar} and \eqref{eq:j_par} we easily show that 
\[         J_{i} \subseteq (J_{i-2} \otimes T(V)_{>0}) \cap (V \otimes R^{(\frac{i-3}{2})} \otimes R).      \]
Using that $J_{i-2} \subseteq V \otimes R^{(\frac{i-3}{2})}$, Corollary \ref{corollary:interincl2} and Lemma \ref{lemma:interj}, 
we see that $J_{i} \subseteq J_{i-2} \otimes R$, which was to be shown. 
The lemma is thus proved. 
\end{proof}

Moreover, the previous lemma can also be strengthen in the next form.
\begin{corollary}
\label{corollary:2recur}
Let $\{ J_{i} \}_{i \in \NN_{0}}$ be the collection of graded vector subspaces of $T(V)$ defined by $J_{0} = k$ and the recursive identities \eqref{eq:j_impar} and \eqref{eq:j_par}. 
Then, for all $i \in \NN_{0}$ and $j \geq i$ except if $i$ is odd and $j$ is even, we have the inclusion 
\[     J_{j} \subseteq  (J_{j-i} \otimes J_{i}) \cap (J_{i} \otimes J_{j-i}).     \] 
On the other hand, if $i \in \NN$ is odd and $j \geq i$ is even, we have the contention
\[     J_{j} \subseteq (T(V) \otimes J_{j-i} \otimes J_{i}) \cap (J_{j-i} \otimes T(V) \otimes J_{i}) \cap (J_{j-i} \otimes J_{i} \otimes T(V)),     \]
and also the same inclusion for $i$ and $j-i$ interchanged. 
\end{corollary}
\begin{proof}
We note that the condition for the first inclusion means that either $i \in \NN_{0}$ is even and $j \geq i$ is arbitrary, or both $i \in \NN$ and $j \geq i$ are odd.

Let us prove the first inclusion of the statement for $i \in \NN_{0}$ even and $j \geq i$. 
If $i = 0$, the inclusion is obvious, so we will assume that $i \geq 2$. 
We shall only show the contention $J_{j} \subseteq J_{j-i} \otimes J_{i}$, for the other is analogous. 
By the previous lemma we have that $J_{j} \subseteq J_{j-i} \otimes R^{(i/2)}$. 
We shall analyse two cases. 
If $j$ is also even, then Lemma \ref{lemma:2recur} also shows that $J_{j} \subseteq R^{((j-i)/2)} \otimes J_{i}$, which in turn implies that $J_{j} \subseteq J_{j-i} \otimes J_{i}$, 
by Corollary \ref{corollary:interincl2}. 
If $j$ is odd, then $J_{j} \subseteq V \otimes J_{j-1}$ by definition, and, since $j-1$ is even, Lemma \ref{lemma:2recur} tells us that $J_{j-1} \subseteq R^{((j-i-1)/2)} \otimes J_{i}$. 
Therefore, $J_{j} \subseteq V \otimes R^{((j-i-1)/2)} \otimes J_{i}$. 
Corollary \ref{corollary:interincl2} now gives the desired contention in the case that $j$ is odd. 

Let us now proceed to prove the first inclusion of the statement for $i \in \NN$ odd and $j \geq i$ also odd, which is the remaining case. 
We note however that this case follows from the previously analysed case when $i \in \NN$ is even and $j \geq i$ is arbitrary, by interchanging $i$ with $j-i$.  

Let us now consider the second inclusion of the statement, which has the assumption that $i \in \NN$ is odd and $j \geq i$ is even. 
In this case, we will first show the contention $J_{j} \subseteq T(V) \otimes J_{j-i} \otimes J_{i}$. 
By \eqref{eq:j_imparbisbis} we see that $J_{j} \subseteq T(V) \otimes J_{j - 1} \otimes V$. 
Since $i-1$ is even, the previous lemma tells us that $J_{j-1} \subseteq J_{j-i} \otimes R^{((i-1)/2))}$, so 
$J_{j} \subseteq T(V) \otimes J_{j-i} \otimes R^{((i-1)/2))} \otimes V$. 
On the other hand, it is trivial to see from the definitions \eqref{eq:j_impar} and \eqref{eq:j_par} that $J_{j} \subseteq T(V) \otimes J_{i}$. 
Now, using the inclusions $J_{i} \subseteq R^{((i-1)/2))} \otimes V$ and $T(V) \otimes J_{j-i} \subseteq T(V)$, together with Corollary \ref{corollary:interincl2}, 
we conclude that $J_{j} \subseteq T(V) \otimes J_{j-i} \otimes J_{i}$, which was to be shown. 
The inclusion $J_{j} \subseteq J_{j-i} \otimes J_{i} \otimes T(V)$ is proved in an analogous manner. 

Finally, we shall prove the inclusion $J_{j} \subseteq J_{j-i} \otimes T(V) \otimes J_{i}$, under the assumption that $i \in \NN$ is odd and $j \geq i$ is even. 
The definitions \eqref{eq:j_impar} and \eqref{eq:j_par} tell us that $J_{j} \subseteq (T(V) \otimes J_{i}) \cap R^{(j/2)}$, and the latter is obviously included in 
\[     (T(V) \otimes J_{i}) \cap (R^{((j-i-1)/2)} \otimes V \otimes T(V) \otimes V \otimes R^{((i-1)/2)}).     \]
In the same manner, we have the inclusion $J_{j} \subseteq (J_{j-i} \otimes T(V)) \cap R^{(j/2)}$, and the latter is obviously included in 
\[     (J_{j-i} \otimes T(V)) \cap (R^{((j-i-1)/2)} \otimes V \otimes T(V) \otimes V \otimes R^{((i-1)/2)}).     \]
By Corollary \ref{corollary:interincl2} we get the desired contention. 
The corollary is thus proved. 
\end{proof}

In order to handle the elements of the graded vector spaces $J_{i}$ we will use the following notation. 
Let us suppose that $W_{1}, \dots, W_{m}$ are graded vector subspaces of $T(V)_{>0}$. 
An element $\omega \in W_{1} \otimes \dots \otimes W_{m} \subseteq T(V)_{>0}$ can be written as a finite sum 
\[       \sum_{(\omega)} \omega_{(1)}| \omega_{(2)}| \dots|\omega_{(i-1)} |\omega_{(m)},     \]
where the elements $\omega_{(i)} \in W_{i}$, for $i = 1, \dots, m$. 
If (say) the vector subspaces $W_{j_{1}}, \dots, W_{j_{l}}$ of $T(V)_{>0}$, for $1 \leq j_{1} < \dots < j_{l} \leq m$, coincide with the graded vector subspace $V \subseteq T(V)_{>0}$, 
we shall typically denote this by writing a bar over each of the factors $\omega_{(j_{1})}, \dots, \omega_{(j_{l})}$. 
Moreover, we shall usually omit the sum, which will be implicitly assumed. 

Let $j \in \NN_{0}$. 
By \eqref{eq:j_par}, we see that, if $\omega \in J_{2 j + 1}$, then we may write it either in the form $\omega = \bar{\omega}_{(1)} | \omega_{(2)}$ 
or in the form $\omega = \omega_{(1)} | \bar{\omega}_{(2)}$, where the sum is implicit as explained before. 
We may thus consider the elements $(\pi \otimes \mathrm{id}_{J_{2 j}}) (\omega) \otimes 1$ and $1 \otimes (\mathrm{id}_{J_{2 j}} \otimes \pi) (\omega)$ in $A \otimes J_{2 j} \otimes A$, 
which may be written as $\pi(\bar{\omega}_{(1)}) | \omega_{(2)} | 1$ and $1 | \omega_{(1)} | \pi(\bar{\omega}_{(2)})$, respectively. 
Note that these elements are uniquely defined by Corollary \ref{corollary:maps}. 
We then get the map $J_{2 j + 1} \rightarrow A \otimes J_{2 j} \otimes A$ given by 
\[     \omega \mapsto \pi(\bar{\omega}_{(1)}) | \omega_{(2)} | 1 - 1 | \omega_{(1)} | \pi(\bar{\omega}_{(2)}),     \]
which we may also simply write as 
\[     \omega \mapsto \bar{\omega}_{(1)} | \omega_{(2)} | 1 - 1 | \omega_{(1)} | \bar{\omega}_{(2)},     \]
using the obvious identification of $V$ inside $A$. 
We shall denote by $\delta_{2 j + 1}^{b} : A \otimes J_{2 i + 1} \otimes A \rightarrow A \otimes J_{2 j} \otimes A$, the $A^{e}$-linear extension of the former. 
Note that it is a morphism of graded $A$-bimodules. 
Moreover, notice that the morphism $\delta_{1}^{b}$ given before coincides with the corresponding differential (denoted in the same way) given in \eqref{eq:resprojminbi}. 

On the other hand, for $j \in \NN$, we consider a map ${}^{j-1}\tilde{T} \rightarrow A \otimes J_{2 j - 1} \otimes A$ of graded vector spaces given by the sum of the maps 
${}^{j-1}\tilde{T}^{N} \rightarrow A \otimes J_{2 j - 1} \otimes A$ of graded vector spaces, for all $N \in \NN_{0}$, which we now define. 
Given $\omega \in {}^{j-1}\tilde{T}^{N}$, then, for all $0 \leq l \leq N$, 
we may write it in the form $\omega = \omega_{(1)}^{l} | \omega_{(2)}^{l} | \omega_{(3)}^{l}$, where $\omega_{(1)}^{l} \in V^{(l)}$, $\omega_{(2)}^{l} \in J_{2 j - 1}$ 
and $\omega_{(3)}^{l} \in V^{(N-l)}$, where the sum is explicit as explained before. 
For each $0 \leq l \leq N$ we may thus consider the element 
\[     (\pi_{|_{V^{(l)}}} \otimes \mathrm{id}_{J_{2 j - 1}} \otimes \pi_{|_{V^{(N-l)}}}) (\omega) = \pi(\omega_{(1)}^{l}) | \omega_{(2)}^{l} | \pi(\omega_{(3)}^{l}),     \] 
in $A \otimes J_{2 j - 1} \otimes A$. 
Again, note that these elements are uniquely defined by Corollary \ref{corollary:maps}. 
We have thus obtained a map 
\[     {}^{j-1}\tilde{T} \rightarrow A \otimes J_{2 j -1} \otimes A     \] 
of graded vector spaces formed by the sum of the maps 
\[     {}^{j-1}\tilde{T}^{N} \rightarrow A \otimes J_{2 j -1} \otimes A     \] 
given by 
\[     \omega \mapsto \sum_{l = 0}^{N} \pi(\omega_{(1)}^{l}) | \omega_{(2)}^{l} | \pi(\omega_{(3)}^{l}).     \]
The inclusion $J_{2 j} \subseteq {}^{j-1}\tilde{T}$ of graded vector spaces in turn induces a map of graded vector spaces 
\[     J_{2 j} \rightarrow A \otimes J_{2 j -1} \otimes A     \] 
We shall denote by $\delta_{2 j}^{b} : A \otimes J_{2 j} \otimes A \rightarrow A \otimes J_{2 j - 1} \otimes A$, the $A^{e}$-linear extension of the former. 
Note that it is a morphism of graded $A$-bimodules. 
Note that $\delta_{2}^{b} : A \otimes J_{2} \otimes A \rightarrow A \otimes J_{1} \otimes A$ coincides with the one given by the same name in \eqref{eq:resprojminbi}. 

We have thus defined a collection of graded $A$-bimodules $\{ K_{L-R} (A)_{i} = A \otimes J_{i} \otimes A \}_{i \in \NN_{0}}$ provided with morphisms of graded $A$-bimodules $\delta_{i}^{b} : K_{L-R} (A)_{i} \rightarrow K_{L-R} (A)_{i-1}$ for $i \in \NN$. 
We have that 
\begin{proposition} 
Let $A$ be a locally finite dimensional nonnegatively graded algebra, and let $\{ J_{i} \}_{i \in \NN_{0}}$ be the collection of graded vector subspaces of $T(V)$ defined by $J_{0} = k$ and the recursive identities \eqref{eq:j_impar} and \eqref{eq:j_par}. 
We consider the collection $\{ K_{L-R} (A)_{i} = A \otimes J_{i} \otimes A \}_{i \in \NN_{0}}$ of graded $A$-bimodules 
provided with morphisms of graded $A$-bimodules $\delta_{i}^{b} : K_{L-R} (A)_{i} \rightarrow K_{L-R} (A)_{i-1}$, for $i \in \NN$, which were defined 
in the three previous paragraphs. 
It gives in fact a complex of graded $A$-bimodules, \textit{i.e.} $\delta_{i-1}^{b} \circ \delta_{i}^{b} = 0$, for $i \in \NN_{\geq 2}$.
\end{proposition}
\begin{proof}
We have thus to prove that $\delta_{i-1}^{b} \circ \delta_{i}^{b} = 0$, for $i \in \NN_{\geq 2}$.
By the exactness of the complex \eqref{eq:resprojminbi}, we get that the previous equality holds for $i=2$. 
Let us suppose that $i \geq 3$. 
We first suppose that $i$ is odd, and we further assume that $\omega \in J_{i}$ is given by $\omega = \sum_{N \in \NN_{0}} \omega^{N} = \sum_{N \in \NN_{0}} \omega'^{N}$, 
where $\omega^{N} \in V \otimes {}^{(i-3)/2}\tilde{T}^{N}$, $\omega'^{N} \in {}^{(i-3)/2}\tilde{T}^{N} \otimes V$, and both sums have finite support. 
Then 
\begin{align*}
     (\delta_{i-1}^{b} \circ \delta_{i}^{b}) (1 | \omega | 1) &= \delta_{i-1}^{b} (\sum_{N \in \NN_{0}} \pi(\bar{\omega}_{(1)}^{N}) | \omega_{(2)}^{N} | 1 
                                                                                                                                                 - 1 | \omega'^{N}_{(1)} | \pi(\bar{\omega}'^{N}_{(2)})) 
     \\
                                                                                    &= \sum_{N \in \NN_{0}} \bar{\omega}_{(1)}^{N} \delta_{i-1}^{b}(\omega_{(2)}^{N})  
                                                                                          - \delta_{i-1}^{b}(\omega'^{N}_{(1)}) \bar{\omega}'^{N}_{(2)}
     \\
                                                                                    &= \sum_{N \in \NN_{0}} \sum_{l=0}^{N} 
                                                                                         \bar{\omega}_{(1)}^{N} \pi(\omega_{(2)}^{N,l}) | \omega_{(3)}^{N,l} | \pi(\omega_{(4)}^{N,l})  
      \\ 
                                                                                     &\phantom{=}- \sum_{N \in \NN_{0}} \sum_{l=0}^{N} 
                                                                                       \pi(\omega'^{N,l}_{(1)}) |  \omega'^{N,l}_{(2)} | \pi(\omega'^{N,l}_{(3)}) \bar{\omega}_{(4)}^{N}.
\end{align*}
By a direct telescopic cancellation, the last member coincides with 
\[          \sum_{N \in \NN_{0}} \bar{\omega}_{(1)}^{N} \pi(\omega_{(2)}^{N,N}) | \omega_{(3)}^{N,N} | 1  
- \sum_{N \in \NN_{0}} 1 |  \omega'^{N,0}_{(2)} | \pi(\omega'^{N,0}_{(3)}) \bar{\omega}_{(4)}^{N},     \] 
where we have set without loss of generality that $\omega_{(4)}^{N,N} = 1$ and $\omega'^{N,0}_{(1)} = 1$.
The sum vanishes, for Lemma \ref{lemma:2recur} tells us that  
\begin{align*}
   \omega &= \sum_{N \in \NN_{0}} \bar{\omega}_{(1)}^{N} \omega_{(2)}^{N,N} | \omega_{(3)}^{N,N} \in R \otimes J_{i - 2},
   \\
   \omega &= \sum_{N \in \NN_{0}} \omega'^{N,0}_{(2)} | \omega'^{N,0}_{(3)} \bar{\omega}_{(4)}^{N} \in J_{i - 2} \otimes R.
\end{align*}
                                                                        
On the other hand, let $i$ be even and and let $\omega \in J_{i}$ be of the form $\omega = \sum_{N \in \NN_{0}} \omega^{N}$, 
where $\omega^{N} \in {}^{(i-2)/2}\tilde{T}^{N}$ and the sum has finite support. 
We have that 
\begin{align*}
     (\delta_{i-1}^{b} \circ \delta_{i}^{b}) (1 | \omega | 1) &= \delta_{i-1}^{b} (\sum_{N \in \NN_{0}} \sum_{l = 0}^{N} \pi(\omega_{(1)}^{N,l}) | \omega_{(2)}^{N,l} | \pi(\omega_{(3)}^{N,l})) 
     \\
                                                                                     &= \sum_{N \in \NN_{0}} \sum_{l = 0}^{N} \pi(\omega_{(1)}^{N,l}) \delta_{i-1}^{b}(\omega_{(2)}^{N,l}) \pi(\omega_{(3)}^{N,l})
     \\
                                                                                     &= \sum_{N \in \NN_{0}} \sum_{l = 0}^{N} \pi(\omega_{(1)}^{N,l}) \bar{\omega}_{(2)}^{N,l} | \omega_{(3)}^{N,l} | \pi(\omega_{(4)}^{N,l})
      \\ 
                                                                                     &\phantom{=}- \sum_{N \in \NN_{0}} \sum_{l = 0}^{N}
                                                                                       \pi(\omega_{(1)}^{N,l}) | \omega_{(2)}^{N,l} | \bar{\omega}_{(3)}^{N,l} \pi(\omega_{(4)}^{N,l}). 
\end{align*}
A telescopic cancellation tells us that the last member coincides with 
\[          \sum_{N \in \NN_{0}} \pi(\omega_{(1)}^{N,N}) \bar{\omega}_{(2)}^{N,N} | \omega_{(3)}^{N,N} | 1  
- \sum_{N \in \NN_{0}} 1 | \omega_{(2)}^{N,0} | \bar{\omega}_{(3)}^{N,0} \pi(\omega_{(4)}^{N,0}),     \]
where we have set without loss of generality that $\omega_{(4)}^{N,N} = 1$ and $\omega_{(1)}^{N,0} = 1$. 
The previous sum vanishes, since Lemma \ref{lemma:2recur} implies that   
\begin{align*}
   \omega &= \sum_{N \in \NN_{0}} \omega_{(1)}^{N,N} \bar{\omega}_{(2)}^{N,N} | \omega_{(3)}^{N,N} \in R \otimes J_{i - 2},
   \\
   \omega &= \sum_{N \in \NN_{0}} \omega^{N,0}_{(2)} | \bar{\omega}^{N,0}_{(3)} \omega_{(4)}^{N,0} \in J_{i - 2} \otimes R.
\end{align*}
This proves the proposition. 
\end{proof}

By the previous statement $(K_{L-R} (A)_{\bullet}, \delta_{\bullet}^{b})$ is a complex of graded $A$-bimodules. 
It can also be regarded as an augmented complex $\delta_{0}^{b} : K_{L-R} (A)_{\bullet} \rightarrow A$ where $\delta_{0}^{b}$ is given by the multiplication of $A$, as in \eqref{eq:resprojminbi}. 
Hence, the augmented complex $K_{L-R} (A)_{\bullet}$ coincides with the minimal projective resolution of $A$ as an $A$-bimodule up to homological degree $2$.  
It will be called the \emph{multi-Koszul bimodule complex} of $A$. 

We also define the (augmented) complex $(K (A)_{\bullet}, \delta_{\bullet})$ (resp., $(K(A)'_{\bullet}, \delta'_{\bullet})$) of graded free left (resp., right) $A$-modules given by 
$(K_{L-R} (A)_{\bullet} \otimes_{A} k, \delta_{\bullet}^{b} \otimes_{A} \mathrm{id}_{k})$ (resp., $(k \otimes_{A} K_{L-R} (A)_{\bullet}, \mathrm{id}_{k} \otimes_{A} \delta_{\bullet}^{b})$), 
which will be called the \emph{left (resp., right) multi-Koszul complex} of $A$. 
We note that in either case the differentials are $A$-linear maps preserving the degree. 
We also remark that the left (resp., right) multi-Koszul complex of $A$ coincides with the minimal projective resolution of the left (resp., right) module $k$ 
seen at the beginning of this section up to homological degree $2$. 

\begin{definition}
\label{definition:multikoszul}
Let $A$ be a locally finite dimensional nonnegatively graded connected algebra.
We say that $A$ is \emph{left (resp., right) multi-Koszul} if the (augmented) left (resp., right) multi-Koszul complex of $A$ defined in the previous paragraphs 
provides a resolution of the trivial left (resp., right) $A$-module $k$, and in this case it is called the \emph{left (resp., right) multi-Koszul resolution} for $A$. 
\end{definition}

\begin{remark}
If $A$ is left (resp., right) multi-Koszul, then the (augmented) left (resp., right) multi-Koszul complex of $A$ is in fact a minimal projective resolution of the trivial left (resp., right) $A$-module $k$. 
Indeed, the left (resp., right) multi-Koszul resolution for $A$ is minimal, by the comments in the antepenultimate paragraph of Section \ref{sec:homo}, 
because the induced differential of the complex $k \otimes_{A} K(A)_{\bullet}$ (resp., $K(A)'_{\bullet} \otimes_{A} k$) vanishes, 
due to \eqref{eq:j_par} and \eqref{eq:j_imparbisbis}. 
It is also straightforward to see that an algebra is left (resp., right) multi-Koszul if and only if its left (resp., right) multi-Koszul complex defined above is acyclic in positive homological degrees. 
\end{remark}

We have the following natural result, following the lines of \cite{BM}, Thm. 4.4. 
\begin{proposition}
\label{proposition:leftbiright}
Let $A$ be a locally finite dimensional nonnegatively graded connected algebra. 
The following statements are equivalent:
\begin{itemize}
 \item[(i)] $A$ is left multi-Koszul.
 \item[(ii)] The augmented multi-Koszul bimodule complex $\delta_{0}^{b} : K_{L-R}(A)_{\bullet} \rightarrow A$ is exact. 
 \item[(iii)] $A$ is right multi-Koszul.  
\end{itemize}
\end{proposition}
\begin{proof}
We shall prove the equivalence of $(i)$ and $(ii)$, the equivalence between $(ii)$ and $(iii)$ being analogous. 
Suppose that $A$ is left multi-Koszul. 
Applying the functor $(\place) \otimes_{A} k$ to the augmented multi-Koszul bimodule complex, we obtain the (augmented) complex $(K(A)_{\bullet},\delta_{\bullet})$, 
which is exact when $A$ is left multi-Koszul. 
Since the $A$-bimodules $K_{L-R}(A)_{i}$ are graded-free and left bounded for all $i\in \mathbb{N}_{0}$,
\cite{BM}, Prop. 4.1, (or \cite{B4}, Lemme 1.6) implies that the augmented complex $\delta_{0}^{b} : K_{L-R}(A)_{\bullet} \rightarrow A$ is exact.

Assume now that the complex $K_{L-R}(A)_{\bullet}$ is exact in positive degrees. 
Since $\delta_{0}^{b} : K_{L-R}(A)_{\bullet} \rightarrow A$ is exact, it is a projective resolution of $A$ in the category of graded left bounded right $A$-modules. 
Since $A$ is a projective right $A$-module, the complex $\delta_{0}^{b} : K_{L-R}(A)_{\bullet} \rightarrow A$ is homotopically trivial as a complex of objects of graded right $A$-modules. 
Therefore, its image under the functor $(\place) \otimes_{A} k$ is \textit{a fortiori} homotopically trivial (as a complex of vector spaces). 
Since this image is the left multi-Koszul complex of $A$, it is exact in positive homological degrees, so $A$ is left multi-Koszul. 
\end{proof}

By the previous results, we shall usually simply say that an algebra $A$ is multi-Koszul, unless we want to emphasize the use of the corresponding complexes. 

\begin{remark}
Suppose that $A$ is left multi-Koszul. 
Since there is an obvious isomorphism of complexes of the form $k \otimes_{A^{e}} K_{L-R}(A)_{\bullet} \simeq k \otimes_{A} K(A)_{\bullet}$, 
having in fact vanishing differential, the comments in the antepenultimate paragraph of Section \ref{sec:homo} tell us that 
the complex $(K_{L-R}(A)_{\bullet}, \delta_{\bullet}^{b})$ is a minimal projective resolution of $A$ in the category of graded $A$-bimodules.
\end{remark}

\begin{remark}
Let $A = T(V)/\cl{R}$ be a locally finite dimensional nonnegatively graded connected algebra. 
Let us we consider a different positive grading on $V$ such that $R$ also remains a graded vector subspace of $T(V)^{\geq 2}$ for the new grading on $V$ 
(\textit{e.g.} we double the grading of $V$). 
Since the condition \eqref{eq:sprel} defining a space of relations is also independent of the grading, we see that $R$ is still a space of relations for the algebra $A$ provided with the new grading. 
This in turn implies that the multi-Koszul complex of $A$ is \emph{grading independent}. 
In particular, $A$ is multi-Koszul for the former grading if and only if it is multi-Koszul for the new grading, 
which we may roughly reformulate as stating that the multi-Koszul property introduced in Definition \ref{definition:multikoszul} is also grading independent. 
\end{remark}

Before proceeding further we would like to state some comments on the previous definition. 
As it may have been noticed, we have used the same terminology (\textit{i.e.} left or right multi-Koszul) as the one introduced in \cite{HeRe}, Def. 3.1 and Rmk. 3.3, 
where the nonnegatively graded connected algebras were further assumed to be finitely generated in degree one and with a finite number of relations. 
We claim that the new definition introduced here coincides with the one considered in \cite{HeRe} if the algebras satisfy the assumptions of the latter article. 

In order to do so, we shall use the following little variation of the notation used in \cite{HeRe}, Def. 3.1. 
Suppose for the moment that $A = T(V)/\cl{R}$, where $V$ is a finite dimensional vector space considered to be concentrated in degree $1$, 
and $R = \oplus_{s \in S} R_{s} \subseteq T(V)_{\geq 2}$, for $S \subseteq \NN_{\geq 2}$, is a finite dimensional graded vector space. 
For each $s \in \mathbb{N}_{\geq 2}$, we recall that the map $n_{s}:\mathbb{N}_{0}\rightarrow \mathbb{N}_{0}$ is given by
\[
n_{s}(2 j)=s j, \hskip 0.4cm n_{s}(2j+1)=s j+1.
\]
If $s\in S$, we will denote
\[     \tilde{J}_{i}^{s} = \bigcap_{l=0}^{n_{s}(i)-s} V^{(l)} \otimes R_{s} \otimes V^{(n_{s}(i)-s-l)},     \]
for $i \geq 2$, and $\tilde{J}_{i}^{s} = V^{(i)}$, for $i = 0, 1$. 
Moreover, we define 
\[      
\tilde{J}_{i} = \bigoplus_{s \in S}  \tilde{J}_{i}^{s}, 
\]
if $i \geq 2$, and $\tilde{J}_{i} = V^{(i)}$, if $i = 0, 1$. 

The differential of the left multi-Koszul complex $(A \otimes \tilde{J}_{\bullet})_{\bullet \in \NN_{0}}$ introduced in \cite{HeRe}, Def. 3.1, will be denoted by $\tilde{\delta}_{\bullet}$. 

\begin{proposition}
\label{proposition:equiv}
Let $A$ be a nonnegatively graded connected algebras, which is further assumed to be finitely generated in degree one and with a finite number of relations, \textit{i.e.} 
$A = T(V)/\cl{R}$, where $V$ is a finite dimensional vector space considered to be concentrated in degree $1$, 
and $R = \oplus_{s \in S} R_{s} \subseteq T(V)_{\geq 2}$, for $S \subseteq \NN_{\geq 2}$, is a finite dimensional graded vector space. 
Then, the left (resp., right) multi-Koszul complex in the Definition \ref{definition:multikoszul} coincides with the left (resp., right) multi-Koszul complex introduced in \cite{HeRe}, Definition 3.1 (resp., Remark 3.3), which in turn implies that $A$ is left (resp., right) multi-Koszul in the sense of Definition \ref{definition:multikoszul} if and only if it is left (resp., right) multi-Koszul in the sense introduced in \cite{HeRe}, Definition 3.1 (resp., Remark 3.3). 
\end{proposition}
\begin{proof}
We shall prove the statement for the left multi-Koszul complexes, the right case being analogous. 
Actually, the equivalence of both right multi-Koszul properties could also follow from the statement for the left case from Proposition \ref{proposition:leftbiright} and \cite{HeRe}, Cor. 3.13. 

Let us first note that the left multi-Koszul complex introduced in Definition \ref{definition:multikoszul} trivially coincides with the one considered in \cite{HeRe}, 
Def. 3.1, up to homological degree $2$. 
Note that $J_{2} = R = \tilde{J}_{2}$. 
They further coincide up to homological degree $3$, since in this case 
\[       J_{3} = (V \otimes R) \cap (R \otimes V) = \bigoplus_{s \in S} (V \otimes R_{s}) \cap (R_{s} \otimes V) = \bigoplus_{s \in S} \tilde{J}_{3}^{s} = \tilde{J}_{3},      \]
where we have used Fact \ref{fact:inter}, since $V$ is concentrated in degree $1$ and $R_{s}$ in degree $s$, 
and the differentials in homological degree $3$ for both complexes clearly coincide. 

Let us suppose that both complexes coincide up to homological degree $d$, \textit{i.e.} we have that $J_{i} = \tilde{J}_{i}$ and $\delta_{i} = \tilde{\delta}_{i}$ for $i \leq d$. 
 
Let us first assume that $d$ is odd. 
We claim that
\begin{equation}
\label{eq:jtildeparimpar}
      \tilde{J}_{2 j + 1} = (V \otimes \tilde{J}_{2 j}) \cap (\tilde{J}_{2 j} \otimes V),     
\end{equation}
for $j \in \NN_{0}$. 
Indeed, if $j = 0$ the previous identity is trivial, whereas for $j \in \NN$ it follows from a rather simple computation
\begin{align*}
     (V \otimes \tilde{J}_{2 j}) \cap (\tilde{J}_{2 j} \otimes V) 
     &= \Big(V \otimes \big(\bigoplus_{s\in S} \tilde{J}_{2 j}^{s}\big)\Big) \cap \Big(\big(\bigoplus_{s\in S} \tilde{J}_{2 j}^{s}\big) \otimes V\Big)
     \\
     &=\Big(\bigoplus_{s\in S} V \otimes \tilde{J}_{2 j}^{s}\Big) \cap \Big(\bigoplus_{s\in S} \tilde{J}_{2 j}^{s} \otimes V\Big)
     \\
     &= \bigoplus_{s\in S} \big((V \otimes\tilde{J}_{2 j}^{s}) \cap (\tilde{J}_{2 j}^{s} \otimes V)\big)
     \\
     &= \bigoplus_{s\in S} \tilde{J}_{2 j + 1}^{s} = \tilde{J}_{2 j +1},
\end{align*}
where we have used Fact \ref{fact:inter} in the antepenultimate equality, taking into account that $V$ is concentrated in degree $1$ 
and each $\tilde{J}_{2 j}^{s}$ is concentrated in degree $n_{s}(2 j) = s j$. 
Now, \eqref{eq:j_par} and \eqref{eq:jtildeparimpar} tell us that $J_{d+1} = \tilde{J}_{d+1}$. 
It is trivial to verify that the differentials also satisfy the identity $\delta_{d+1} = \tilde{\delta}_{d+1}$. 

On the other hand, let us suppose that $d$ is even. 
In this case we claim that 
\begin{equation}
\label{eq:jtildeimparpar}
      \tilde{J}_{2 j} = \big(\bigoplus_{N \in \NN}  \bigcap_{\begin{scriptsize}\begin{matrix}\bar{n} \in \mathrm{Par}(j-1) \\ \bar{m} \in \mathrm{Par}_{j}(N) \end{matrix}\end{scriptsize}} 
V^{(m_{1})} \otimes \tilde{J}_{2 n_{1}} \otimes \dots \otimes V^{(m_{j-1})} \otimes \tilde{J}_{2 n_{j-1}} \otimes V^{(m_{j})}  \big) \cap R^{(j)},     
\end{equation}
for all $j \in \NN$, which we prove as follows. 
It is trivially verified for $j = 1$, since the right term in that case is just $T(V)_{>0} \cap R$, which coincides with $\tilde{J}_{2} = R$. 

Let us consider $j \geq 2$. 
We first show that the right member of \eqref{eq:jtildeimparpar} contains the left one. 
Indeed, note that 
\[    V^{(m_{1})} \otimes \tilde{J}_{2 n_{1}} \otimes \dots \otimes V^{(m_{j-1})} \otimes \tilde{J}_{2 n_{j-1}} \otimes V^{(m_{j})}     \]
trivially includes
\[    V^{(m_{1})} \otimes \tilde{J}^{s}_{2 n_{1}} \otimes \dots \otimes V^{(m_{j-1})} \otimes \tilde{J}^{s}_{2 n_{j-1}} \otimes V^{(m_{j})},     \]
for all $s \in S$, so the right member of \eqref{eq:jtildeimparpar} includes 
\[    \big(\bigoplus_{N \in \NN}  \bigcap_{\begin{scriptsize}\begin{matrix}\bar{n} \in \mathrm{Par}(j-1) \\ \bar{m} \in \mathrm{Par}_{j}(N) \end{matrix}\end{scriptsize}} 
V^{(m_{1})} \otimes \tilde{J}^{s}_{2 n_{1}} \otimes \dots \otimes V^{(m_{j-1})} \otimes \tilde{J}^{s}_{2 n_{j-1}} \otimes V^{(m_{j})}  \big) \cap R^{(j)},     
\]
for all $s \in S$, which trivially coincides with 
\[    \big(\bigoplus_{N \in \NN_{> s(j-2)}} \bigcap_{l=0}^{N} 
V^{(l)} \otimes R_{s} \otimes V^{(N-l)} \big) \cap R^{(j)},    \]
for all $s \in S$. 
By Fact \ref{fact:inter}, the latter space is given by 
\[    \bigoplus_{N \in \NN_{> s(j-2)}} \big(\bigcap_{l=0}^{N} 
V^{(l)} \otimes R_{s} \otimes V^{(N-l)} \cap R^{(j)}_{N+s}\big),    \]
for all $s \in S$. 

On the other hand, note that $(R^{(j)})_{n}$ is the direct sum of the independent vector subspaces $\{R_{s_{1}} \otimes \dots \otimes R_{s_{j}} \}_{s_{1} + \dots + s_{j} = n}$. 
This follows from a simple recursive argument using Corollary \ref{corollary:interind0}. 
In particular, $\omega \in (R^{(j)})_{n}$ has to be written in a unique manner as a sum of unique elements 
$\omega_{s_{1}, \dots, s_{j}} \in R_{s_{1}} \otimes \dots \otimes R_{s_{j}}$, where we assume that $s_{1} + \dots + s_{j} = n$.  
It is thus trivial that $\cap_{l=0}^{N} 
V^{(l)} \otimes R_{s} \otimes V^{(N-l)}$ is included in $R^{(j)}_{N+s}$ if $N = s (j-1)$. 
Moreover, another application of Corollary \ref{corollary:interind} tells us that the intersection of them vanishes otherwise. 
This in turn implies that 
\[    \bigoplus_{N \in \NN_{>s(j-2)}} \big(\bigcap_{l=0}^{N} 
V^{(l)} \otimes R_{s} \otimes V^{(N-l)} \cap R^{(j)}_{N+s}\big),    \]
coincides with $\tilde{J}^{s}_{2 j}$, for all $s \in S$, and we then conclude that the right member of \eqref{eq:jtildeimparpar} contains the left one.

We will now prove that the left member of \eqref{eq:jtildeimparpar} contains the right one. 
In order to do so, note first that the latter member is trivially 
included in
\begin{align*}
    \big(\bigoplus_{N \in \NN} \bigcap_{l=0}^{N} V^{(l)} \otimes &\tilde{J}_{2 j - 1} \otimes V^{(N-l)}\big) \cap R^{(j)} 
     \\
     &= \big(\bigoplus_{N \in \NN} \bigcap_{l=0}^{N} \bigoplus_{s \in S} V^{(l)} \otimes \tilde{J}_{2 j - 1}^{s} \otimes V^{(N-l)}\big) \cap R^{(j)}
     \\
     &= \big(\bigoplus_{N \in \NN} \bigoplus_{s \in S} \bigcap_{l=0}^{N} V^{(l)} \otimes \tilde{J}_{2 j - 1}^{s} \otimes V^{(N-l)}\big) \cap R^{(j)}   
     \\
     &= \big(\bigoplus_{N \in \NN} \bigoplus_{s \in S} \bigcap_{l=0}^{N+n_{s}(2 j -3)} V^{(l)} \otimes R_{s} \otimes V^{(N+n_{s}(2 j -3)-l)}\big) \cap R^{(j)},
\end{align*}
where we used Fact \ref{fact:inter} in the second equality, for $V$ is concentrated in degree $1$ 
and $\tilde{J}_{2 j - 1}^{s}$ is concentrated in degree $n_{s}(2 j -1) = s j - s + 1$. 
Using Fact \ref{fact:inter} once more in the last member, we get that the latter should be the direct sum of the intersection of the $n$-th homogeneous components of each corresponding term, 
\textit{i.e.} the intersection of the $n$-th direct summand of $R^{(j)}$ and 
\begin{equation}
\label{eq:intern}
     \bigoplus_{s \in S_{j,n}} \bigcap_{l=0}^{n-s} V^{(l)} \otimes R_{s} \otimes V^{(n-s-l)} 
\end{equation}
where $S_{j,n} = \{ s \in S : n_{s}(2 j -1) < n \}$. 
Note that if $s \in S_{j,n}$, then $n > s$, for $j \geq 2$. 
As explained in the previous paragraph, a direct application of Corollary \ref{corollary:interind0} tells us that the former intersection should coincide with 
\[     \bigoplus_{s \in S} \big(\bigcap_{l=0}^{n-s} (V^{(l)} \otimes R_{s} \otimes V^{(n-s-l)}) \cap (R_{s}^{(j)})_{n}\big),     \] 
which is directly seen to be equal to  
\[       \bigoplus_{s \in S} \big(\bigcap_{l=0}^{n-s} (V^{(l)} \otimes R_{s} \otimes V^{(n-s-l)}) \cap (R_{s}^{(j)})_{n}\big) 
          = \bigoplus_{s \in S} (\tilde{J}_{2 j}^{s})_{n} = (\tilde{J}_{2 j})_{n}.     \]
This proves that the left member of \eqref{eq:jtildeimparpar} contains the right one, and so the equality of the assertion \eqref{eq:jtildeimparpar} holds.

Finally, \eqref{eq:j_imparbis} and \eqref{eq:jtildeimparpar} imply that $J_{d+1} = \tilde{J}_{d+1}$. 
It is also direct to check that the differentials $\delta_{d+1}$ and $\tilde{\delta}_{d+1}$ coincide. 
The proposition is thus proved. 
\end{proof}

\begin{remark}
Since the previous definition of left or right multi-Koszul property coincides with the one considered in \cite{HeRe} if the algebras satisfy the assumptions of the latter article, 
by \cite{HeRe}, Rmk. 3.4, we get that it also coincides with the corresponding one given in \cite{B2}, Def. 2.10 (or \cite{BDW}, Section 5), if the algebra is homogeneous. 
Thus, a homogeneous algebra is left (resp., right) multi-Koszul if and only if it is generalized left (resp., right) Koszul. 
\end{remark}

\begin{example}
\label{example:sym} 
We will provide a collection of examples of multi-Koszul algebras, 
which was one of the main motivations of this article. 
The space of generators $V$ of these graded algebras does not lie in degree $1$, 
so they cannot be considered as multi-Koszul algebras for the definition given in \cite{HeRe} in any natural manner. 

Given two nonnegative integers $n, s \in \NN_{0}^{2} \setminus \{ (0,0) \}$, and a collection of symmetric $(s \times s)$-matrices $(\Gamma_{a,b}^{i})$, 
for $i = 1, \dots, n$ ($a,b = 1, \dots, s$), 
the (associative) \emph{super Yang-Mills algebra} $\mathrm{YM}(n,s)^{\Gamma}$ over an algebraically closed field $k$ of characteristic zero is defined as follows. 
Take $V = V_{2} \oplus V_{3}$ be a graded vector space over $k$, where $\dim_{k}(V_{2}) = n$ and $\dim_{k}(V_{3}) = s$, 
and choose in fact a (homogeneous) basis $\mathcal{B} = \mathcal{B}_{2} \cup \mathcal{B}_{3}$ of $V$, where $\mathcal{B}_{2} = \{ x_{1}, \dots, x_{n} \}$ 
and $\mathcal{B}_{3} =\{ z_{1}, \dots, z_{s} \}$, with $|x_{i}| = 2$, for all $i = 1, \dots, n$, and $|z_{a}| = 3$, for all $a = 1, \dots, s$, 
We suppose further that the matrices $(\Gamma^{i}_{a,b})$ satisfy the nondegeneracy assumption explained in the third paragraph before Rmk. 1 of \cite{Her10}. 

The graded algebra $\mathrm{YM}(n,s)^{\Gamma}$ is given by the quotient of the graded free algebra $T(V)$, 
by the homogeneous relations given by  
\begin{align*}
 r_{0,i} &= \sum_{j=1}^{n} [x_{j},[x_{j},x_{i}]] - \frac{1}{2} \sum_{a,b = 1}^{s} \Gamma^{i}_{a,b} [z_{a},z_{b}], 
 \\
 r_{1,a} &= \sum_{i=1}^{n} \sum_{b = 1}^{s} \Gamma_{a,b}^{i} [x_{i},z_{b}],  
\end{align*} 
for $i = 1, \dots, n$ and $a = 1, \dots, s$, respectively. 
The bracket $[, \hskip 0.6mm ]$ denotes the graded commutator, \textit{i.e.} $[a,b] = a b - (-1)^{|a| |b|} b a$, for $a, b \in \mathcal{B}$. 
They have been previously considered by M. Movshev and A. Schwarz in \cite{MS06} (see also the preprint article \cite{Mov} by M. Movshev). 

Using the explicit description of the minimal projective resolution of the trivial module $k$ over the super Yang-Mills algebra $\mathrm{YM}(n,s)^{\Gamma}$ 
(for $(n,s) \neq (1,0), (1,1)$) given in \cite{Her10}, Prop. 2, and Corollary \ref{corollary:consezero} given below, we see that these graded algebras are multi-Koszul. 
\end{example}

\subsection{\texorpdfstring{An equivalent description of multi-Koszul algebras}{An equivalent description of multi-Koszul algebras}}
\label{subsec:multikoszul}

We would like to make some comments on the left multi-Koszul complex of $A$. 
The obvious statements for the right multi-Koszul complex trivially hold. 
First, given $i \in \NN_{0}$, note that the map of graded vector spaces $J_{i+1} \rightarrow \Ker(\delta_{i})$ given by the restriction of $\delta_{i+1}$ is injective. 
This can be proved as follows. 
The cases $i =0, 1$ are immediate, so we will suppose that $i \geq 2$.  
Furthermore, the kernel of the restriction of $\delta_{i+1}$ to $J_{i+1}$ is easily seen to be $J_{i+1} \cap (I \otimes J_{i})$. 
The first term of this intersection is included in $R \otimes J_{i-1}$, by Lemma \ref{lemma:2recur}, whereas the second is included in $I \otimes V \otimes J_{i-1}$ if $i$ is odd, by \eqref{eq:j_par}, 
and it is included in $I \otimes T(V)_{>0} \otimes J_{i-1}$ if $i$ is even, by \eqref{eq:j_imparbisbis}. 
In both situations we have thus that $I \otimes J_{i} \subseteq I \otimes T(V)_{>0} \otimes J_{i-1}$. 
Then, the kernel of the restriction of $\delta_{i+1}$ to $J_{i+1}$ is contained in the intersection 
\[
     (R \otimes J_{i-1}) \cap (I \otimes T(V)_{>0} \otimes J_{i-1}) = (R \cap (I \otimes T(V)_{>0})) \otimes J_{i-1},          
\]
where we have used Proposition \ref{proposition:imp} and Lemma \ref{lemma:interj}. 
By the defining property \eqref{eq:sprel} of the space of relations the last space vanishes, and we thus get that the restriction of $\delta_{i+1}$ to $J_{i+1}$ is injective. 
This proves the claim. 

For each $i \in \NN_{0}$, let us now consider the map $J_{i+1} \rightarrow k \otimes_{A} \Ker(\delta_{i})$ given by the composition of $J_{i+1} \rightarrow \Ker(\delta_{i})$ 
and the canonical projection $\Ker(\delta_{i}) \rightarrow k \otimes_{A} \Ker(\delta_{i})$. 
We claim that this composition is in fact injective if $i$ is even. 
This can be proved as follows. 
By the comments in the last paragraph of Section \ref{sec:homo}, we know that the mentioned map is in fact an isomorphism for $i = 0$ (and also for $i = 1$). 
We shall suppose thus that $i \geq 2$. 
Since $i$ is even, the image of the map $J_{i+1} \rightarrow \Ker(\delta_{i})$ is contained in $V \otimes J_{i}$, 
so one sees that the kernel of $J_{i+1} \rightarrow k \otimes_{A} \Ker(\delta_{i})$ 
vanishes if and only if $J_{i} \cap \Ker(\delta_{i}) = 0$, which follows from the previous paragraph. 

We have thus the following result. 
\begin{lemma}
\label{lemma:imp}
Let $A = T(V)/\cl{R}$ be a locally finite dimensional nonnegatively graded algebra and let $(K(A)_{\bullet},\delta_{\bullet})$ be its left multi-Koszul complex. 
Given $i \in \NN_{0}$, the map of graded vector spaces $J_{i+1} \rightarrow \Ker(\delta_{i})$ given by the restriction of $\delta_{i+1}$ is injective. 
Consider now the map of graded vector space given by the composition of the previous morphism and the canonical projection $\Ker(\delta_{i}) \rightarrow k \otimes_{A} \Ker(\delta_{i})$. 
If $i$ is even or $i = 1$, it is injective. 
If $i$ is odd and $i \geq 3$, we note that the restriction to the $n$-th homogeneous components of the previous composition map is injective 
if there are no nontrivial homogeneous components of $(\Ker(\delta_{i}))_{m}$ for $m < n$.
\end{lemma}

The corresponding formulation of the lemma for the right multi-Koszul complex of $A$ is obvious, and we shall refer to the lemma whether we are considering the left or the right version. 
We may in fact prove one of the main result of this section (\textit{cf.} \cite{HeRe}, Prop. 3.12):
\begin{proposition}
\label{proposition:tresimp}
Let $A = T(V)/\cl{R}$ be a locally finite dimensional nonnegatively graded algebra. 
Then $A$ is left (resp., right) multi-Koszul if and only if there is an isomorphism of graded vector spaces 
$\Tor_{i}^{A}(k,k) \simeq J_{i}$, for all $i \in \NN_{0}$.
\end{proposition}
\begin{proof}
We shall prove the statement for the left multi-Koszul property, since the right one is analogous. 
Moreover, we will only show the ``if'' part, since the converse follows immediately from the minimality of the (left) multi-Koszul complex. 

Assume the existence of the isomorphism of graded vector spaces in the statement. 
We will prove that the left multi-Koszul complex is in fact a minimal projective resolution of the trivial left $A$-module $k$. 
In fact, we will show that $K(A)_{\bullet}$ is a minimal projective resolution of $k$ up to homological degree $i$, for all $i \in \NN$. 
Since the former coincides with such a minimal projective resolution up to homological degree $2$, we suppose that the statement is true for $i \geq 2$. 
By the comments on the construction of projective covers in Section \ref{sec:homo}, and the assumption $\Tor_{i+1}^{A}(k,k) \simeq J_{i+1}$, 
there is an isomorphism of graded vector spaces $\bar{h}_{i+1} : J_{i+1} \rightarrow k \otimes_{A} \Ker(\delta_{i})$. 

If $i \geq 2$ and $i$ is even, the previous lemma tells us that the composition
\begin{equation}
\label{eq:deltaibar}
J_{i+1} \hookrightarrow \Ker(\delta_{i})\twoheadrightarrow k \otimes_{A} \Ker(\delta_{i}),
\end{equation}
where the first map is the restriction of $\delta_{i+1}$, is injective.
Hence, the composition of this map with the inverse of $\bar{h}_{i+1}$ is an injective endomorphism of graded vector spaces of $J_{i+1}$, 
so an isomorphism, since the latter is a locally finite dimensional graded vector space. 
Hence, the map \eqref{eq:deltaibar} is also an isomorphism, so $\delta_{i+1}$ is in fact a projective cover of $\Ker(\delta_{i})$. 

We now assume that $i \geq 2$ and $i$ is odd. 
We consider as before the map of graded vector spaces given by \eqref{eq:deltaibar}, which we denote by $f_{i+1}$. 
Let us denote by $(f_{i+1})_{n} : (J_{i+1})_{n} \rightarrow (k \otimes_{A} \Ker(\delta_{i}))_{n}$ the restriction to the $n$-th homogeneous components, and the same for $(\bar{h}_{i+1})_{n}$. 
We shall prove that $(f_{i+1})_{n}$ is an isomorphism for all $n \in \NN_{0}$ by induction on $n$. 
Note that, by the isomorphism $\Tor_{i+1}^{A}(k,k) \simeq J_{i+1}$ of graded vector spaces, it suffices to prove that $(f_{i+1})_{n}$ is injective, for an injective map between finite dimensional vector spaces is automatically an isomorphism, since the corresponding $n$-th homogeneous components are finite dimensional. 
Let $n_{\mathrm{min}} \in \NN$ be the first positive integer such that $(J_{i+1})_{n_{\mathrm{min}}} \neq 0$. 
Note that this in particular implies that $(f_{i+1})_{n}$ is injective for $n < n_{\mathrm{min}}$. 
The assumption $\Tor_{i+1}^{A}(k,k) \simeq J_{i+1}$ implies that $\Ker(\delta_{i})$ is concentrated in degrees greater that or equal to $n_{\mathrm{min}}$. 
This in turn implies that $(f_{i+1})_{n_{\mathrm{min}}}$ is injective by Lemma \ref{lemma:imp}. 
Let us thus assume that $(f_{i+1})_{n}$ is injective for $n \leq m$, for some $m \geq n_{\mathrm{min}}$,  
so, again by the hypothesis $\Tor_{i+1}^{A}(k,k) \simeq J_{i+1}$, they should be in fact isomorphisms, since the corresponding $n$-th homogeneous components are finite dimensional. 
We shall prove that $(f_{i+1})_{m+1}$ is also injective. 
If this is not the case, by the definition of the map $(f_{i+1})_{m+1}$ and the inductive assumption, we must have that the intersection $(J_{i+1})_{m+1} \cap (T(V)_{>0} \otimes J_{i+1})_{m+1}$ is nontrivial. 
However, since $J_{i+1} \subseteq R^{(i+1)/2}$, the previous intersection is included in $R^{(i+1)/2} \cap (T(V)_{>0} \otimes R^{(i+1)/2})$, which vanishes by the 
defining property \eqref{eq:sprel} of the space of relations. 
The proposition is thus proved. 
\end{proof}

We now have the following immediate consequence of Proposition \ref{proposition:tresimp} and the isomorphism \eqref{eq:isotorext}.
\begin{proposition}
\label{proposition:tresimpext}
For a locally finite dimensional nonnegatively graded algebra connected $A$ with space of relations $R$, 
the following statements are equivalent:
\begin{itemize}
\item[\textit{(i)}] $A$ is multi-Koszul,
\item[\textit{(i)}] $\mathcal{E}xt_{A}^{i} (k,k) \simeq J_{i}^{\#}$, for all $i \in \NN_{0}$, where $(\place)^{\#}$ denotes the graded dual of a graded vector space.
\end{itemize}
\end{proposition}

\begin{remark}
\label{remark:extext}
Suppose further that $A$ is a finitely generated multi-Koszul algebra with a finite dimensional space of relations. 
Then, the multi-Koszul resolution for $A$ is composed of finitely generated projective $A$-modules, for each vector space $J_{i}$ is finite dimensional, 
so, by the comments in the penultimate paragraph of Section \ref{sec:homo}, there is a canonical identification $\mathcal{E}xt_{A}^{\bullet} (k,k) \simeq \Ext_{A}^{\bullet} (k,k)$. 
\end{remark}

We may also mention an easy corollary of the previous Lemma. 
\begin{corollary}
\label{corollary:consezero}
Let $A = T(V)/\cl{R}$ be a locally finite dimensional nonnegatively connected graded algebra. 
Suppose that its left multi-Koszul complex $(K(A)_{\bullet},\delta_{\bullet})$ is exact in homological degrees $\bullet = 1, \dots, N-1$, for $N \in \NN_{\geq 2}$, and that $\delta_{N}$ 
is injective. 
Then $A$ is left multi-Koszul. 
\end{corollary}
\begin{proof}
By the proof of Proposition \ref{proposition:tresimp}, the exactness hypothesis is equivalent to say that $\Tor_{i}^{A}(k,k) \simeq J_{i}$, for all $i = 1, \dots, N-1$. 
Moreover, the injectivity of $\delta_{N}$ implies that that $J_{N+1}$ vanishes, for the restriction of $\delta_{N+1}$ to $J_{N+1}$ is injective, by Lemma \ref{lemma:imp}, and its image lies in the Kernel of $\delta_{N}$. 
This in turn implies that $J_{i}$ also vanishes, for $i \geq N+1$, 
by the definitions \eqref{eq:j_impar} and \eqref{eq:j_par}. 
In consequence, the left multi-Koszul complex $(K(A)_{\bullet},\delta_{\bullet})$ is exact in positive degrees, so $A$ is left multi-Koszul. 
\end{proof}

\subsection{\texorpdfstring{Properties of multi-Koszul algebras}{Properties of multi-Koszul algebras}}
\label{subsec:propmultikoszul}

We have a direct consequence of the Proposition \ref{proposition:leftbiright}. 
Consider the (unique) anti-morphism of (unitary) algebras $\tau : T(V) \rightarrow T(V)$ such that $\tau|_{V} = \mathrm{id}_{V}$. 
It is in fact an anti-automorphism of $T(V)$, and it further induces an algebra anti-isomorphism $\bar{\tau}: A \rightarrow T(V)/\cl{\tau (R)}$. 
In other words, it induces an isomorphism between the (usual) opposite algebra $A^{\mathrm{op}}$ of $A$ and $A^{\circ} = T(V)/\cl{\tau(R)}$. 
Note that $\tau(R) \subseteq T(V)^{\geq 2}$ is clearly a space of relations of $A^{\circ}$. 
By the previous isomorphism, we may thus say that $A^{\circ}$ is (also) the \emph{opposite algebra} of $A$.
\begin{corollary}
\label{corollary:koszulitydual}
The algebra $A^{\circ}$ is multi-Koszul if and only if $A$ is multi-Koszul.
\end{corollary}
\begin{proof}
It is an immediate consequence of Proposition \ref{proposition:leftbiright}. 
\end{proof}

Since the length of a minimal projective resolution of $k$ gives the global dimension of $A$, the following proposition is an immediate consequence of Corollary \ref{corollary:consezero}.
\begin{corollary}
Let $A$ be a locally finite dimensional nonnegatively graded algebra. 
If the global dimension of $A$ is $2$, then $A$ is multi-Koszul.
\end{corollary}

We also have the following result, which is a generalization of \cite{HeRe}, Prop. 3.7, and which shows that the multi-Koszul property 
is stable under free products. 
The proof is completely parallel but we provide it for completeness. 
\begin{proposition}
\label{proposition:libre}
Let $\{ B^{s} : s \in S \}$, where $S$ is an index set, be a finite collection of locally finite dimensional nonnegatively graded connected algebras such that $B^{s}$ is multi-Koszul, for each $s \in S$. 
Then, the free product (\textit{i.e.} the coproduct in the category of graded algebras) $A = \coprod_{s \in S} B^{s}$ of the collection $\{ B^{s} : s \in S \}$ is a multi-Koszul algebra. 
\end{proposition}
\begin{proof}
Suppose $B^{s} = T(V^{s})/\cl{R^{s}}$, for $s \in S$, is a multi-Koszul algebra, where $V^{s}$ and $R^{s} \subseteq T(V^{s})_{\geq 2}$ are 
locally finite dimensional positively graded vector spaces. 
By the definition of the free product, we may consider that $A = T(V)/\cl{R}$, where $V = \oplus_{s \in S} V^{s}$ and $R = \oplus_{s \in S} R^{s}$. 
The canonical inclusion $B^{s} \hookrightarrow A$ is a morphism of graded algebras, and it makes $A$ a free graded (left or right) $B^{s}$-module. 
For $s \in S$, denote by $J_{i}^{s}$ the graded vector space defined by the recursive equations \eqref{eq:j_impar} and \eqref{eq:j_par} for the algebra $B^{s}$, 
and by $J_{i}$ the corresponding one defined for $A$. 

Since the graded vector spaces $V^{s}$ are independent, by the definition of the tensor algebra it is trivial to verify that $J_{i} = \oplus_{s \in S} J_{i}^{s}$, for $i \in \NN$.     
If $(K(B^{s})_{\bullet},\delta_{\bullet}^{s})$ is the multi-Koszul complex of $B^{s}$, which is acyclic in positive homological degrees by assumption, 
we have that $A \otimes_{B^{s}} K(B^{s})_{\bullet} = A \otimes J_{\bullet}^{s}$ is also acyclic in positive homological degrees, as we now show. 
By the K\"unneth spectral sequence $E^{2}_{p,q} = \Tor_{p}^{B^{s}}(A,H_{q}(K(B^{s})_{\bullet} )) \Rightarrow H_{p+q}(A \otimes_{B^{s}} K(B^{s})_{\bullet})$ 
(see \cite{W}, Application 5.6.4). 
The exactness of the Koszul complex of $B^{s}$ and the freeness of the $B^{s}$-module $A$ imply that $E^{2}_{p,q} = 0$ if $(p,q) \neq (0,0)$, 
so $H_{n}(A \otimes_{B^{s}} K(B^{s})_{\bullet}) = 0$, for $n \geq 1$. 

We now note that the multi-Koszul complex $(K(A)_{\bullet},\delta_{\bullet})$ of the algebra $A$ can be decomposed as $K(A)_{\bullet} = \oplus_{s \in S} A \otimes_{B^{s}} K(B^{s})_{\bullet}$, 
for $\bullet \geq 1$, and $\delta_{\bullet} = \oplus_{s \in S} \delta_{\bullet}^{s}$ for $\bullet \geq 2$, the exactness of $A \otimes_{B^{s}} K(B^{s})_{\bullet}$ 
in positive homological degrees tells us that $K(A)_{\bullet}$ is acyclic in homological degrees greater than or equal to $2$. 
On the other hand, the exactness of the multi-Koszul complex in homological degree $1$ is automatically satisfied for a nonnegatively graded connected algebra. 
We have thus that $K(A)_{\bullet}$ is exact in positive homological degrees, so $A$ is multi-Koszul. 
\end{proof}

Another interesting property for this class of algebras is the following. 
\begin{proposition}
\label{proposition:k2}
Let $A$ be a finitely generated nonnegatively graded connected algebra such that its space of relations $R$ is finite dimensional, and assume that $A$ is multi-Koszul. 
Then, the graded algebra $\mathcal{E}xt_{A}^{\bullet}(k,k) = \mathrm{Ext}_{A}^{\bullet}(k,k)$ is generated by $\Ext_{A}^{1} (k,k)$ and $\Ext_{A}^{2} (k,k)$, \textit{i.e.} it is $\mathcal{K}_{2}$ 
(in the sense of Cassidy and Shelton). 
\end{proposition}
\begin{proof}
Let $E_{i}$ (following the notation of the article \cite{CS}, Section 4) be the matrix with entries in 
$I/(T(V)_{>0} . I + I . T(V)_{>0}) \simeq R$ given by the class of the matrix with entries in $T(V)$ 
which is a lift of the composition $\delta_{i-1} \circ \delta_{i}$.  
By Lemma \ref{lemma:2recur} we have the inclusion $J_{i} \subseteq R \otimes J_{i-2}$, which implies that the matrix $E_{i}$ 
represents an injective linear transformation of the form $T(V) \otimes J_{i} \rightarrow T(V) \otimes J_{i-1}$, 
so the rows $E_{i}$ are linearly independent over $k$. 
The statement is now a direct consequence of \cite{CS}, Thm. 4.4. 
\end{proof}

\begin{remark}
\label{remark:malo}
As explained in \cite{HeRe}, Rmk. 3.24, the converse of the previous proposition is not true in general 
(\textit{e.g.} see the algebra $B$ in \cite{CG}, which is not multi-Koszul, but it is $2$-$3$-Koszul in the sense of \cite{GM}). 
\end{remark}

\section{\texorpdfstring{The $A_{\infty}$-algebra structure of the Yoneda algebra of a multi-Koszul algebra}{The A-infinity-algebra structure of the Yoneda algebra of a multi-Koszul algebra}}
\label{sec:yonedamultikoszul}

In this section, we shall provide a direct procedure to compute the complete $A_{\infty}$-algebra structure of the Yoneda algebra $\mathcal{E}xt_{A}^{\bullet}(k,k)$ of a multi-Koszul algebra $A$. 
We will first compute the (plain) algebra structure by explicitly providing quasi-isomorphisms in both directions between the cochain complexes $\mathcal{H}om_{A}(K(A)_{\bullet},k)$ and $\mathcal{E}nd_{A}(K(A)_{\bullet})$. 
Then we shall compute the remaining higher multiplications.  
In order to do so, it will be useful to profit from the theory of $A_{\infty}$-algebras and $A_{\infty}$-coalgebras. 
Even though we refer for further references to \cite{LH}, or \cite{Prou}, we will provide a short introduction, in particular for stating our (sign) conventions and notation. 
A more intensive study of the $A_{\infty}$-algebra structure of the Yoneda algebra may be found in \cite{LPWZ09}, to which we also refer for further reading. 

We recall that, since the cochain complex $\mathcal{E}nd_{A}(K(A)_{\bullet})$ is a differential graded algebra (or dg algebra for short), 
its cohomology has an algebra structure, and further a structure of minimal $A_{\infty}$-algebra, defined via the theorem of T. Kadeishvili in \cite{K82}. 
Following \cite{LPWZ09}, any of these $A_{\infty}$-algebra structures on the cohomology of $\mathcal{E}nd_{A}(K(A)_{\bullet})$ is called a \emph{model}. 
Kadeishvili also proved in the mentioned article that any of these models on the cohomology ring are in fact quasi-isomorphic. 
It can be shown that the algebra structure of $\mathcal{E}xt_{A}^{\bullet}(k,k)$ is in fact independent of the projective resolution of the trivial module $k$ used to compute it, 
and furthermore, the $A_{\infty}$-structure is also unique up to quasi-isomorphism (see \cite{LPWZ09}, Lemma 4.2, (a)). 
Moreover, the endomorphism dg algebra of any projective resolution of the trivial $A$-module $k$ is quasi-isomorphic to the graded dual $B^{+}(A)^{\#}$ of the bar construction of $A$, 
as $A_{\infty}$-algebras (see \cite{LPWZ09}, Lemma 4.2, (b)).

\subsection{{\texorpdfstring{The algebra structure}{The algebra structure}}}

We will first compute the Yoneda product for the cohomology space $\mathcal{E}xt_{A}^{\bullet}(k,k)$. 
Just for convenience we recall that $\mathcal{E}nd_{A}(K(A)_{\bullet})$ is the graded algebra whose $i$-th homogeneous component $\mathcal{E}nd_{A}^{i}(K(A)_{\bullet})$ is given by 
\[     \prod_{j \in \ZZ} \mathcal{H}om_{A}(K(A)_{j},K(A)_{j-i}),     \]
for $i \in \ZZ$, together with the multiplication induced by the composition and the differential $\partial$ given as follows. 
If $(f_{j})_{j \in \ZZ} \in \mathcal{E}nd_{A}^{i}(K(A)_{\bullet})$, then the $j$-th component of $\partial((f_{j})_{j \in \ZZ}) \in \mathcal{E}nd_{A}^{i+1}(K(A)_{\bullet})$ is given by 
\[     \delta_{j-i} \circ f_{j} - (-1)^{i} f_{j-1} \circ \delta_{j},     \]
where we remark that, by definition of the complex $(K(A)_{\bullet},\delta_{\bullet})$, $\delta_{l}$ vanishes for $l \leq 0$. 

On the one hand, the augmentation map $\delta_{0} : K(A)_{\bullet} \rightarrow k$ induces a quasi-isomorphism, that we will denote $\phi$, 
from $\mathcal{E}nd_{A}(K(A)_{\bullet})$ to $\mathcal{H}om_{A}(K(A)_{\bullet},k)$. 
It is given explicitly by the linear map which sends $(f_{j})_{j \in \ZZ} \in \mathcal{E}nd_{A}^{i}(K(A)_{\bullet})$ to $(-1)^{i} \delta_{0} \circ f_{i}$. 
We remark that the differential of $\mathcal{H}om_{A}(K(A)_{\bullet},k)$ is given by $g \mapsto (-1)^{i} g \circ \delta_{i+1}$, for $g \in \mathcal{H}om_{A}(K(A)_{i},k)$. 
It is trivial to see that it is a map of complexes, \textit{i.e.} it commutes with the differentials, and furthermore it respects the grading.
That this is a quasi-isomorphism is a standard fact on homological algebra (see \cite{Bour}, \S 5.2, Prop. 4; and the first paragraph of \S 7.1, where it is done for 
the endomorphism dg algebra of an injective resolution, but the analogous considerations hold for projective resolutions). 
Note that the procedure described in the paragraph can be applied to any projective resolution of $k$. 

We will now show how to construct a linear map from $\mathcal{H}om_{A}(K(A)_{\bullet},k)$ to $\mathcal{E}nd_{A}(K(A)_{\bullet})$, that will be denoted by $\psi$. 
We first note that, since the differential of the complex $\mathcal{H}om_{A}(K(A)_{\bullet},k)$ vanishes, we have the obvious identification $\mathcal{H}om_{A}(K(A)_{i},k) \simeq J_{i}^{\#}$
of graded vector spaces, where $(\place)^{\#}$ denotes the usual graded dual of a graded vector space. 
Let us consider an element $f \in \mathcal{H}om_{A}(K(A)_{i},k)$, where $i \in \NN_{0}$, and we denote by $\bar{f} \in J_{i}^{\#}$ the corresponding element of the graded dual. 
We will construct elements $f_{j} \in \mathcal{H}om_{A}(K(A)_{j},K(A)_{j-i})$, for all $j \in \ZZ$, such that $(f_{j})_{j \in \ZZ}$ lies in fact in the kernel of the differential $\partial$ of 
the dg algebra $\mathcal{E}nd_{A}(K(A)_{\bullet})$ and such that $(-1)^{i} \delta_{0} \circ f_{i} = f$.
In order to do so, we will consider two cases. 

If the nonnegative integer $i$ is even, then we define $f_{j} : K(A)_{j} \rightarrow K(A)_{j-i}$ as follows. 
If $j < i$, we set $f_{j} = 0$.
We now consider the case $j \geq i$. 
We first note that, by Corollary \ref{corollary:2recur}, $J_{j} \subseteq J_{j-i} \otimes J_{i}$. 
Hence, we have a map of graded vector spaces $p_{j,i}^{f} : J_{j} \rightarrow J_{j-i}$ defined as the composition of the inclusion 
$J_{j} \subseteq J_{j-i} \otimes J_{i}$ and the map $\mathrm{id}_{J_{j-i}} \otimes \bar{f}$.
We then set $f_{j}$ to be $A$-linear map given by $f_{j}(a \otimes \omega) = a \otimes p_{j,i}^{f}(\omega)$. 

If the nonnegative integer $i$ is odd, the map $f_{j} : K(A)_{j} \rightarrow K(A)_{j-i}$ is given as follows. 
If $j < i$, we also set $f_{j} = 0$, so let us consider the case $j \geq i$. 
If $j$ is odd, Corollary \ref{corollary:2recur} tells us that $J_{j} \subseteq J_{j-i} \otimes J_{i}$, so we consider again the map $p_{j,i}^{f} : J_{j} \rightarrow J_{j-i}$ of graded vector spaces 
given by the composition of the inclusion $J_{j} \subseteq J_{j-i} \otimes J_{i}$ and the map $\mathrm{id}_{J_{j-i}} \otimes \bar{f}$. 
We then set $f_{j}$ to be $A$-linear map defined as $f_{j}(a \otimes \omega) = - a \otimes p_{j,i}^{f}(\omega)$. 
If $j$ is even, by Corollary \ref{corollary:2recur} we have that $J_{j} \subseteq T(V) \otimes J_{j-i} \otimes J_{i}$. 
We have in this case the map $p_{j,i}^{f} : J_{j} \rightarrow A \otimes J_{j-i}$ of graded vector spaces defined as the composition of the inclusion $J_{j} \subseteq T(V) \otimes J_{j-i} \otimes J_{i}$ 
and $\pi \otimes \mathrm{id}_{J_{j-i}} \otimes \bar{f}$, where $\pi : T(V) \rightarrow A$ is the canonical projection. 
Set $f_{j}$ to be $A$-linear map given by $f_{j}(a \otimes \omega) = a \, p_{j,i}^{f}(\omega)$. 

It is straightforward, but rather tedious, to show that the element $(f_{j})_{j \in \ZZ}$ is in the kernel of the differential of $\mathcal{E}nd_{A}(K(A)_{\bullet})$ 
and such that $(-1)^{i} \delta_{0} \circ f_{i} = f$. 
We define thus the map $\psi : \mathcal{H}om_{A}(K(A)_{\bullet},k) \rightarrow \mathcal{E}nd_{A}(K(A)_{\bullet})$ via $\psi(f) = (f_{j})_{j \in \ZZ}$. 
It is trivially verified that this is a morphism of graded vector spaces, which commutes with the differentials (for the image of $\psi$ lies in the kernel of the differential of its codomain 
and the domain has vanishing differential), and the composition $\phi \circ \psi$ is the corresponding identity map. 
Since $\phi$ is a quasi-isomorphism, $\psi$ satisfies the same property. 
Furthermore, the map induced by $\psi$ at the level of cohomology spaces is in fact the inverse of the corresponding map induced by $\phi$. 

These maps allow us to explicitly compute the algebra structure of the Yoneda algebra $\mathcal{E}xt_{A}^{\bullet}(k,k)$, since given two elements in the Yoneda algebra 
which are represented by $f \in \mathcal{H}om_{A}(K(A)_{i},k)$ and $g \in \mathcal{H}om_{A}(K(A)_{i'},k)$, or more concretely, by $f \in J_{i}^{\#}$ and $g \in J_{i'}^{\#}$, 
the product is just given by $\psi(\phi(f) \phi(g)) \in \mathcal{H}om_{A}(K(A)_{i+i'},k)$, or simply by the induced element in $J_{i+i'}^{\#}$. 
This can be written down in the following very explicit manner. 
In order to do so, we recall that we will consider $\mathcal{E}xt_{A}^{\bullet}(k,k)$ as a (cohomological) bigraded algebra 
(\textit{i.e.} a cohomological bigraded vector space provided with a multiplication, unit and augmentation which respect both gradings), with one grading coming from the cohomological degree, 
which will be called the \emph{cohomological grading}, and another one coming from the original grading of the modules over $A$, which will be called the \emph{Adams grading}. 

In the same manner, we consider that the space $J_{i}$ is concentrated in homological degree $i$ and the Adams grading coincides with the one induced by the grading of $V$, 
which was the only one we considered before. 
We moreover define the homological bigraded vector space $J = \oplus_{i \in \NN_{0}} J_{i}$, with the homological and Adams gradings induced from the ones of the homological bigraded vector spaces $J_{i}$, for $i \in \NN_{0}$. 
Note that $J^{\#} = \oplus_{i \in \NN_{0}} J_{i}^{\#}$ is a cohomological bigraded vector space, where $J_{i}^{\#}$ is concentrated in cohomological degree $i$. 
As explained before, we will apply the Koszul sign rule to the cohomological grading but not to the Adams grading. 

Let $i$ and $i'$ be two nonnegative integers. 
If either $i$ or $i'$ is even, we shall denote by $\iota_{i,i'}$ the inclusion map $J_{i+i'} \subseteq J_{i} \otimes J_{i'}$ given in Corollary \ref{corollary:2recur}. 
If both $i$ and $i'$ are odd, we shall denote by $\iota_{i,i'}$ the composition of the inclusion map $J_{i+i'} \subseteq T(V) \otimes J_{i} \otimes J_{i'}$ given in Corollary \ref{corollary:2recur} 
together with the map $\epsilon_{T(V)} \otimes \mathrm{id}_{J_{i}} \otimes \mathrm{id}_{J_{i'}}$, where $\epsilon_{T(V)} : T(V) \rightarrow k$ denotes the augmentation of the tensor algebra. 
In any case, $\iota_{i,i'}$ is a map of graded vector spaces from $J_{i+i'}$ to $J_{i} \otimes J_{i'}$. 
This in turn induces a map of graded vector spaces $\iota_{i,i'}^{\#} : (J_{i} \otimes J_{i'})^{\#} \rightarrow J_{i+i'}^{\#}$. 
Recall that, given $V$ and $W$ two locally finite dimensional bigraded vector spaces, where the first one is the cohomological grading and the second one is the Adams grading, 
we denote by $c_{V,W} : V^{\#} \otimes W^{\#} \rightarrow (V \otimes W)^{\#}$ 
the obvious isomorphism of bigraded vector spaces $c_{V,W}(f \otimes g)(v \otimes w) = (-1)^{\deg(v) \deg(g)} f(v) g(w)$, where $f, g, v, w$ are homogeneous elements, and 
$\deg(\place)$ denotes the cohomological degree of the corresponding element.  
Now, by making use of the usual identification $\mathcal{E}xt_{A}^{i}(k,k) \simeq J_{i}^{\#}$ of cohomological bigraded vector spaces and the previous comments, 
we obtain the following result, which provides a description of the structure of \emph{augmented cohomological bigraded algebra} 
(\textit{i.e.} a cohomological bigraded vector space provided with a multiplication, unit and augmentation which respect both gradings) of the Yoneda algebra.  
\begin{theorem}
\label{theorem:yoneda}
Let $A$ be a multi-Koszul algebra, and let $\{ J_{i} \}_{i \in \NN_{0}}$ be the collection of graded vector subspaces of $T(V)$ defined by $J_{0} = k$ and the recursive identities 
\eqref{eq:j_impar} and \eqref{eq:j_par}. 
We utilize the usual identification $\mathcal{E}xt_{A}^{i}(k,k) \simeq J_{i}^{\#}$, for $i \in \NN_{0}$, coming from the use of the multi-Koszul resolution for $A$. 
Given $i,i' \in \NN_{0}$, let $\iota_{i,i'} : J_{i+i'} \rightarrow J_{i} \otimes J_{i'}$ be the map of graded vector spaces introduced in the previous paragraph, and 
$c_{J_{i},J_{i'}} : J_{i}^{\#} \otimes J_{i'}^{\#} \rightarrow (J_{i} \otimes J_{i'})^{\#}$ be the usual identification of graded vector spaces also recalled in the previous paragraph. 
Then, the restriction of the Yoneda product of $\mathcal{E}xt_{A}^{\bullet}(k,k)$ to $J_{i}^{\#} \otimes J_{i'}^{\#}$ is given by $\iota_{i,i'}^{\#} \circ c_{J_{i},J_{i'}}$.    
\end{theorem}

\begin{remark}
It is easy to check that the algebra algebra structure for $\mathcal{E}xt_{A}^{\bullet}(k,k)$ given in the theorem coincides with the one deduced from \cite{HeRe}, Prop. 3.21 and Rmk. 3.22, 
under the further assumptions that the algebra $A$ is finitely generated in degree $1$ with a finite dimensional space of relations. 
In particular, the previous algebra structure also coincides with the one given in \cite{BM}, Prop. 3.1, for the Yoneda algebra of a generalized Koszul (homogeneous) algebra.
\end{remark}

\subsection{{\texorpdfstring{The $A_{\infty}$-algebra structure}{The A-infinity-algebra structure}}}

We will now describe the $A_{\infty}$-algebra structure of the Yoneda algebra. 
First, we note that the graded dual of a locally finite dimensional augmented cohomological bigraded algebra is a \emph{coaugmented homological bigraded coalgebra} 
(\textit{i.e.} a homological bigraded vector space provided with a comultiplication, counit and coaugmentation which respect both gradings). 
Indeed, if $E = \oplus_{(n,m) \in \ZZ^{2}} E^{n}_{m}$ is a locally finite dimensional augmented cohomological bigraded algebra, then $E^{\#}$ is a coaugmented homological bigraded coalgebra, 
where the comultiplication of $E^{\#}$ is the composition of the graded dual of the multiplication of $E$ and $c_{E,E}^{-1}$, the counit of $E^{\#}$ is the graded dual of the unit of $E$, and 
the coaugmentation of $E^{\#}$ is the graded dual of the augmentation of $E$. 
The corresponding result given by interchanging the terms ``augmented cohomological bigraded algebra'' and ``coaugmented homological bigraded coalgebra'' 
(without the locally finite dimensional assumption) also holds, 
with the obvious analogous definitions of mulplication, unit and augmentation. 
Hence, one finds that the homological bigraded vector space $J$ is naturally a coaugmented homological bigraded coalgebra. 
The counit of $J$ is the canonical projection onto the $J_{0}$ and the unit is the inclusion of $J_{0}$ inside $J$.
Moreover, the comultiplication $\Delta$ of $J$ is given by $\Delta|_{J_{i}} = \sum_{l=0}^{i} \iota_{l,i-l}$. 

We will extend the coaugmented coalgebra structure on $J$ to a \emph{minimal Adams graded coaugmented $A_{\infty}$-coalgebra} on $J$. 
For the following definitions we refer to \cite{Prou}, Chapitre 3, Section 3.1 (or also \cite{LH}, D\'ef. 1.2.1.1, 1.2.1.8, using the obvious equivalences 
between non(co)unitary objects and (co)augmented ones), even though we do not follow the same sign conventions and they do not consider any Adams grading 
(see for instance \cite{LPWZ09} for several uses of Adams grading in $A_{\infty}$-algebra theory). 
Moreover, our definitions will be somehow more restricted with respect to the gradings that the usual ones. 
We first recall that an \emph{Adams graded augmented $A_{\infty}$-algebra} structure on a cohomological bigraded vector space $A$ 
such that the cohomological grading is nonnegative and the component of zero cohomological degree is just $k$ (also lying in Adams degree zero) is the following data: 
\begin{itemize}
\item[(i)] a collection of maps $m_{i} : A^{\otimes i} \rightarrow A$ for $i \in \NN$ of cohomological degree $2-i$ 
and Adams degree zero satisfying the \emph{Stasheff identities} given by   
\begin{equation}
\label{eq:ainftyalgebra}
   \sum_{(r,s,t) \in \mathcal{I}_{n}} (-1)^{r + s t}  m_{r + 1 + t} \circ (\mathrm{id}_{A}^{\otimes r} \otimes m_{s} \otimes \mathrm{id}_{A}^{\otimes t}) = 0,
\end{equation} 
for $n \in \NN$, where $\mathcal{I}_{n} = \{ (r,s,t) \in \NN_{0} \times \NN \times \NN_{0} : r + s + t = n \}$. 
\item[(ii)] a map $\eta : k \rightarrow A$ of bidegree $(0,0)$ such that 
\[     m_{i} \circ (\mathrm{id}_{A}^{\otimes r} \otimes \eta \otimes \mathrm{id}_{A}^{\otimes t})     \]
vanishes for all $i \neq 2$ and all $r, t \geq 0$ such that $r+1+t = i$, and 
\[     m_{2} \circ (\mathrm{id}_{A} \otimes \eta) = \mathrm{id}_{A} = m_{2} \circ (\eta \otimes \mathrm{id}_{A}).    \]
\item[(ii)] a map $\epsilon : A \rightarrow k$ of bidegree $(0,0)$ such that $\epsilon \circ \eta = \mathrm{id}_{k}$, $\epsilon \circ m_{2} = \epsilon^{\otimes 2}$, and 
$\epsilon \circ m_{i} =0$, for all $i \in \NN \setminus \{ 2 \}$. 
\end{itemize}
It is further called \emph{minimal} if $m_{1}$ vanishes. 

We recall that a family of linear maps $\{f_{i} : C \rightarrow C_{i}\}_{i \in \NN}$, where $C$ and $C_{i}$, for $i \in \NN$, are vector spaces, is called \emph{locally finite} if, for all $c \in C$, 
there exists a finite subset $S \subseteq \NN$, which depends on $c$, such that $f_{i}(c)$ vanishes for all $i \in \NN \setminus S$. 
An \emph{Adams graded coaugmented $A_{\infty}$-coalgebra} structure on a homological bigraded vector space $C$ such that the homological grading 
is nonnegative and the component of zero homological degree is just $k$ (also lying in Adams degree zero) is the following data: 
\begin{itemize}
\item[(i)] a locally finite collection of maps $\Delta_{i} : C \rightarrow C^{\otimes i}$ for $i \in \NN$ of homological degree $i-2$ and Adams degree zero satisfying the following identities  
\begin{equation}
\label{eq:ainftycoalgebra}
   \sum_{(r,s,t) \in \mathcal{I}_{n}} (-1)^{r s + t}  (\mathrm{id}_{C}^{\otimes r} \otimes \Delta_{s} \otimes \mathrm{id}_{C}^{\otimes t}) \circ \Delta_{r + 1 + t} = 0,
\end{equation}
for $n \in \NN$. 
\item[(ii)] a map $\epsilon : C \rightarrow k$ of bidegree $(0,0)$ such that 
\[     (\mathrm{id}_{C}^{\otimes r} \otimes \epsilon \otimes \mathrm{id}_{C}^{\otimes t}) \circ \Delta_{i}     \]
vanishes for all $i \neq 2$ and all $r, t \geq 0$ such that $r+1+t = i$, and 
\[     (\mathrm{id}_{C} \otimes \epsilon) \circ \Delta_{2} = \mathrm{id}_{C} = (\epsilon \otimes \mathrm{id}_{C}) \circ \Delta_{2}.    \]
\item[(ii)] a map $\eta : k \rightarrow C$ of bidegree $(0,0)$ such that $\epsilon \circ \eta = \mathrm{id}_{k}$, $\Delta_{2} \circ \eta(1) = \eta(1)^{\otimes 2}$, 
and $\Delta_{i} \circ \eta(1) = 0$, for all $i \in \NN \setminus \{ 2 \}$. 
\end{itemize}
An Adams graded coaugmented $A_{\infty}$-coalgebra $C$ is called \emph{minimal} if $\Delta_{1} = 0$. 
Note that the condition that the family $\{ \Delta_{n} \}_{n \in \NN}$ is locally finite follows from the other data if we further suppose that $\Ker(\epsilon)$ is positively graded for the Adams degree. 

We will not introduce the definitions of morphisms between (augmented) $A_{\infty}$-algebras, neither for (coaugmented) $A_{\infty}$-coalgebras, 
and refer to \cite{Prou}, Sections 3.2 and 3.3 (D\'ef. 3.3, 3.4, and 3.11), or \cite{LH}, Sections 1.2 and 1.3. 
We remark however that these morphisms are further supposed to preserve the Adams degree (\textit{cf}. \cite{LPWZ09}, Section 2). 

Notice that the graded dual $C^{\#}$ (as homological bigraded spaces) of an Adams graded coaugmented $A_{\infty}$-coalgebra $C$ is an Adams graded augmented $A_{\infty}$-algebra. 
Indeed, the unit of $C^{\#}$ is the graded dual of the counit of $C$ and the augmentation of $C^{\#}$ is the dual of the coaugmentation of $C$. 
Moreover, for $n \in \NN$, the multiplication $m_{n}$ of $C^{\#}$ is the composition of the canonical map $c_{C,\dots, C} : (C^{\#})^{\otimes n} \rightarrow (C^{\otimes n})^{\#}$ 
and the graded dual of $\Delta_{n}$. 

We shall now proceed to describe the $A_{\infty}$-coalgebra structure of $J$. 
In order to do so, we first construct the following maps. 
Let $n \geq 3$ and $\bar{j} = (j_{1},\dots, j_{n}) \in \NN^{n}_{0}$. 
We write $j_{l} = 2 j'_{l} + r_{l}$, for all $l = 1, \dots, n$, where $j'_{l} \in \NN_{0}$ and $r_{l} \in \{ 0,1\}$, and consider the direct sum decomposition 
\begin{equation}
\label{eq:sp}
     \bigoplus_{N \in \NN_{\geq 2}} \big(\bigcap_{\bar{m} \in \mathrm{Par}_{n+1}(N)} 
V^{(m_{1})} \otimes J_{2 j'_{1}} \otimes \dots \otimes V^{(m_{n})} \otimes J_{2 j'_{n}} \otimes V^{(m_{n+1})}\big).  
\end{equation}
Let $p_{N}^{\bar{j}}$ be the projection of the former vector subspace of the tensor algebra onto the $N$-th direct summand, for $N \in \NN_{\geq 2}$. 

If $\bar{j}$ satisfies that $j_{l}$ is even for all $1 \leq l \leq n$ except for two integers $1 \leq a < b \leq n$, 
for which $j_{a}, j_{b}$ are odd, \eqref{eq:j_imparbisbis} tells us that $J_{j_{1} + \dots + j_{n}}$ is included in \eqref{eq:sp}. 
Moreover, it is trivial to see (using definition \eqref{eq:j_par}, Lemma \ref{lemma:interj} and Proposition \ref{proposition:imp}) that the $n$-th direct summand of the former vector space is 
contained in $J_{2 j'_{1} + 1} \otimes \dots \otimes J_{2 j'_{n} + 1}$. 
Let us denote by $\iota_{2 j_{1} + 1, \dots, 2 j_{n} + 1}$ the map of graded vector spaces from $J_{j_{1} + \dots + j_{n}}$ to $J_{2 j'_{1} + 1} \otimes \dots \otimes J_{2 j'_{n} + 1}$ 
given by the composition of the inclusion of $J_{j_{1} + \dots + j_{n}}$ in \eqref{eq:sp}, the projection $p_{n}^{\bar{i}}$, and the inclusion of the $n$-th direct summand of \eqref{eq:sp} in 
$J_{2 j'_{1} + 1} \otimes \dots \otimes J_{2 j'_{n} + 1}$. 
In other words, given $\bar{i} = (i_{1},\dots,i_{n}) \in \NN^{n}$ such that $i_{l}$ is odd for all $1 \leq l \leq n$, we have defined a map $\iota_{i_{1}, \dots, i_{n}}$ of graded vector spaces 
from $J_{i_{1} + \dots + i_{n}+2-n}$ to $J_{i_{1}} \otimes \dots \otimes J_{i_{n}}$. 
Suppose further that $a=1$, that is, $j_{1}$ is odd. 
Then, the $N$-th direct summand of \eqref{eq:sp} is in fact contained in the $N$-th direct summand of 
\[          \bigoplus_{N \in \NN_{\geq 2}} \big(\bigcap_{\bar{m} \in \mathrm{Par}_{n}(N-1)} 
 J_{j_{1}} \otimes V^{(m_{1})} \otimes J_{2 j'_{2}} \otimes \dots \otimes V^{(m_{n-1})} \otimes J_{2 j'_{n}} \otimes V^{(m_{n})}\big),      \]
for all $N \in \NN_{\geq 2}$, which coincides with 
\[          J_{j_{1}} \otimes \bigoplus_{N \in \NN_{\geq 2}} \big(\bigcap_{\bar{m} \in \mathrm{Par}_{n}(N-1)} V^{(m_{1})} \otimes J_{2 j'_{2}} \otimes \dots \otimes V^{(m_{n-1})} \otimes J_{2 j'_{n}} \otimes V^{(m_{n})}\big),      \]
by Proposition \ref{proposition:imp} and Lemma \ref{lemma:interj}. 
Hence, we see that the restriction of $p_{n}^{\bar{j}}$ to \eqref{eq:sp} coincides with the restriction of $\mathrm{id}_{J_{j_{1}}} \otimes p_{n-1}^{(j_{2},\dots,j_{n})}$ to the same space. 
On the other hand, if we assume that $b = n$, \textit{i.e.} $j_{n}$ is odd, the restriction of $p_{n}^{\bar{j}}$ to \eqref{eq:sp} 
coincides with the restriction of $p_{n-1}^{(j_{1},\dots,j_{n-1})} \otimes \mathrm{id}_{J_{j_{n}}}$ to the same space. 

\begin{remark}
\label{remark:util}
Note that if $\bar{i} = (1, \dots, 1) \in \NN^{n}$, $\iota_{1, \dots, 1}$ is just the composition of the inclusion of $R$ inside the tensor algebra with the canonical projection $\pi_{n} : T(V) \rightarrow V^{(n)}$.  
More generally, for $i \in \NN$ even, consider $\bar{i} = (i-1,1,\dots, 1) \in \NN^{n}$ (the last $n-1$ integers are $1$'s). 
It is easy to check that $\iota_{\bar{i}} : J_{i} \rightarrow J_{i-1} \otimes V^{(n-1)}$ coincides with the composition of the inclusion $J_{i} \subseteq J_{i-1} \otimes T(V)_{>0}$ with the map $\mathrm{id}_{J_{i-1}} \otimes \pi_{n-1}$. 
Indeed, for the case $\bar{i} = (i-1,1,\dots, 1)$, the direct sum decomposition \eqref{eq:sp} gives 
\[    \bigoplus_{N \in \NN_{\geq 2}} \big(\bigcap_{m = 0}^{N} V^{(m)} \otimes J_{i-2} \otimes V^{(N-m)}\big).     \]
Moreover, the $N$-th direct summand of the former is trivially included in the $N$-th direct summand of 
\[    \bigoplus_{N \in \NN_{\geq 2}} J_{i-1} \otimes V^{(N-1)},     \]
for all $N \in \NN_{\geq 2}$, by Proposition \ref{proposition:imp} and Lemma \ref{lemma:interj}. 
Hence, the restriction of $\mathrm{id}_{J_{i-1}} \otimes \pi_{n-1}$ to the first direct sum decomposition coincides with the projection $p_{n}^{\bar{i}}$, 
so the restriction of $\mathrm{id}_{J_{i-1}} \otimes \pi_{n-1}$ to $J_{i}$ coincides with $\iota_{\bar{i}}$. 
\end{remark}

The $A_{\infty}$-(co)algebra structure of $J^{\#}$ (resp., $J$) is described in the next statement. 
\begin{proposition}
\label{proposition:yonedaainf}
Let $A$ be a multi-Koszul algebra, and let $\{ J_{i} \}_{i \in \NN_{0}}$ be the collection of graded vector subspaces of $T(V)$ defined by $J_{0} = k$ and the recursive identities 
\eqref{eq:j_impar} and \eqref{eq:j_par}. 
Define $J = \oplus_{i \in \NN_{0}} J_{i}$, which is considered as a homological bigraded vector space for the grading coming from the index $i$. 
By the previous theorem, it has the structure of a coaugmented bigraded coalgebra, and we denote $\Delta_{2}$ the comultiplication coming from the multiplication given in the previous theorem. 
We consider that $\Delta_{1}$ vanishes. 
Given $n \geq 3$ and $\bar{i} = (i_{1},\dots,i_{n}) \in \NN^{n}$ such that $i_{l}$ is odd for all $1 \leq l \leq n$, let 
$\iota_{i_{1}, \dots, i_{n}} : J_{i_{1} + \dots + i_{n}+2-n} \rightarrow J_{i_{1}} \otimes \dots \otimes J_{i_{n}}$ be the map of graded vector spaces introduced in the paragraph before Remark \ref{remark:util}, 
and $\Delta_{n}|_{J_{i}} = \sum \iota_{i_{1}, \dots, i_{n}}$, where the sum is indexed over all $n$-tuples of odd integers $(i_{1}, \dots, i_{n})$ such that $i_{1} + \dots + i_{n} +2 - n = i$. 
Then $J$ is a minimal Adams graded coaugmented $A_{\infty}$-coalgebra. 
This in turn implies that $J^{\#}$ is a minimal Adams graded augmented $A_{\infty}$-algebra, where the multiplication $m_{n} : (J^{\#})^{\otimes n} \rightarrow J^{\#}$ is the map whose restriction to $J_{i_{1}}^{\#} \otimes \dots \otimes J_{i_{n}}^{\#}$ is zero, if there is $1 \leq l \leq n$ such that $i_{l}$ is even, and is $\iota_{i_{1}, \dots, i_{n}}^{\#} \circ c_{J_{i_{1}}, \dots, J_{i_{n}}}$ otherwise.
\end{proposition}
\begin{proof}
We first note that $\Delta_{n}$ has bidegree $(n-2,0)$, for all $n \in \NN$. 
Moreover, it is easy to check that the family $\{ \Delta_{n} \}_{n \in \NN}$ is locally finite, since $\Ker(\epsilon)$ is positively graded for the Adams degree. 

Since $\Delta_{1}$ is zero, the first, second and third of the identities \eqref{eq:ainftycoalgebra} simply mean that $J$ is a coassociative coalgebra for the comultiplication $\Delta_{2}$, 
which is a direct consequence of the previous theorem. 

Suppose now that $n \geq 3$. 
The assumption that $\Delta_{n}$ vanishes on $J_{i}$ if $i$ is odd, and that its image is included in the sum of tensors of $n$ factors, each of which has odd homological degree, tell us that the 
$(n+1)$-th defining identity \eqref{eq:ainftycoalgebra} simplifies to give 
\begin{equation}
\label{eq:si(n+1)}
     \sum_{r=1}^{n-1} (-1)^{r} (\mathrm{id}_{J}^{\otimes r} \otimes \Delta_{2} \otimes \mathrm{id}_{J}^{\otimes (n-r-1)}) \circ \Delta_{n} = (\mathrm{id}_{J} \otimes \Delta_{n}) \circ \Delta_{2} - (-1)^{n} (\Delta_{n} \otimes \mathrm{id}_{J}) \circ \Delta_{2}.     
\end{equation}
It suffices to show the previous identity restricted to each subspace $J_{i}$. 
Note that the image of each term of the previous identity must be thus in the direct sum of (independent) subspaces $J_{i_{1}} \otimes \dots \otimes J_{i_{n+1}}$, 
where $i = i_{1} + \dots + i_{n+1} + 2 - n$. 
Then, it suffices to prove \eqref{eq:si(n+1)} for the special case that we restrict to $J_{i}$ and we compose with the projection $p_{i_{1}, \dots, i_{n+1}}$ onto $J_{i_{1}} \otimes \dots \otimes J_{i_{n+1}}$, 
where $i = i_{1} + \dots + i_{n+1} + 2 - n$. 
We shall refer to this new identity also as the $(i_{1}, \dots, i_{n+1})$-\emph{specialization} of \eqref{eq:si(n+1)}. 
By the assumption on the higher multiplications $m_{n}$, for $n \geq 3$, the image of any of the terms of \eqref{eq:si(n+1)} must be in the direct sum of the subspaces    
$J_{i_{1}} \otimes \dots \otimes J_{i_{n+1}}$, where there exists at most one $1 \leq j \leq n+1$ such that $i_{j}$ is even. 
If we compose the restriction of \eqref{eq:si(n+1)} to $J_{i}$ with the projection onto $J_{i_{1}} \otimes \dots \otimes J_{i_{n+1}}$, the only cases left to consider are thus: 
\begin{itemize}
\item[(i)] all $i_{1}, \dots, i_{n+1}$ are odd,
\item[(ii)] $i_{1}$ is even and $i_{2}, \dots, i_{n+1}$ are odd, 
\item[(iii)] $i_{n}$ is even and $i_{1}, \dots, i_{n}$ are odd, 
\item[(iv)] there is $1 < j < n+1$ such that $i_{j}$ is the only even integer among $i_{1}, \dots, i_{n+1}$. 
\end{itemize}
Let us write how the composition of $(n+1)$-th defining identity \eqref{eq:si(n+1)} restricted to $J_{i}$ and the projection onto $J_{i_{1}} \otimes \dots \otimes J_{i_{n+1}}$ further simplifies in each case, and how this identity follows from the construction of the higher comultiplication $\Delta_{n}$. 

In case $(i)$, the $(i_{1}, \dots, i_{n+1})$-specialization of \eqref{eq:si(n+1)} simplifies to give 
\[     p_{i_{1}, \dots, i_{n+1}} \circ (\mathrm{id}_{J} \otimes \Delta_{n}) \circ \Delta_{2}|_{J_{i}} = (-1)^{n} p_{i_{1}, \dots, i_{n+1}} \circ (\Delta_{n} \otimes \mathrm{id}_{J}) \circ \Delta_{2}|_{J_{i}},     \]
for all the terms of the left member of \eqref{eq:si(n+1)} vanish. 
By the Koszul sign rule this identity is equivalent to the commutativity of the square 
\[
\xymatrix
{
J_{i_{1} + \dots + i_{n+1}+2-n}
\ar[rr]^{\iota_{i_{1},i_{2} + \dots + i_{n+1}+2-n}}
\ar[dd]^{\iota_{i_{1} + \dots + i_{n}+2-n, i_{n+1}}}
&
&
J_{i_{1}} \otimes J_{i_{2} + \dots + i_{n+1}+2-n}
\ar[dd]^{\mathrm{id}_{J_{i_{1}}} \otimes \iota_{i_{2}, \dots, i_{n+1}}}
\\
&
&
\\
J_{i_{1} + \dots + i_{n}+2-n} \otimes J_{i_{n+1}}
\ar[rr]^{\iota_{i_{1}, \dots, i_{n}} \otimes \mathrm{id}_{J_{i_{n+1}}}}
&
&
J_{i_{1}} \otimes \dots \otimes J_{i_{n+1}}
}
\]
Note that, since all $i_{1}, \dots, i_{n+1}$ are odd the upper horizontal and the left vertical maps are the plain inclusions. 
The commutativity easily follows from the definitions. 
Indeed, consider the direct sum decomposition of the form 
\begin{equation}
\label{eq:sp*}
     \bigoplus_{N \in \NN_{\geq 2}} \big(\bigcap_{\bar{m} \in \mathrm{Par}_{n+2}(N)} 
V^{(m_{1})} \otimes J_{i_{1}-1} \otimes \dots \otimes V^{(m_{n+1})} \otimes J_{i_{n+1}-1} \otimes V^{(m_{n+2})}\big),  
\end{equation}
which trivially includes $J_{i_{1} + \dots + i_{n+1}+2-n}$. 
Proposition \ref{proposition:imp} and Lemma \ref{lemma:interj} tell us that the $N$-th direct summand of the former is trivially included in the $N$-th direct summand of 
\[     \bigoplus_{N \in \NN_{\geq 2}} \big(\bigcap_{\bar{m} \in \mathrm{Par}_{n}(N-2)} \hskip -10pt
J_{i_{1}} \otimes V^{(m_{1})} \otimes J_{i_{2}-1} \otimes V^{(m_{2})} \otimes \dots \otimes V^{(m_{n-1})} \otimes J_{i_{n}-1} \otimes V^{(m_{n})} \otimes J_{i_{n+1}}\big),     \]
for all $N \in \NN_{\geq 2}$, which, also by Proposition \ref{proposition:imp} and Lemma \ref{lemma:interj}, coincides with 
\[     \bigoplus_{N \in \NN_{\geq 2}} J_{i_{1}} \otimes \big(\bigcap_{\bar{m} \in \mathrm{Par}_{n}(N-2)} 
V^{(m_{1})} \otimes J_{i_{2}-1} \otimes \dots \otimes J_{i_{n}-1} \otimes V^{(m_{n})}\big) \otimes J_{i_{n+1}}.     \]
Hence, by the comments at the end of the paragraph before Remark \ref{remark:util} we see that the restriction of 
$\mathrm{id}_{J_{i_{1}}} \otimes p_{n-1}^{(i_{2},\dots,i_{n})} \otimes \mathrm{id}_{J_{i_{n+1}}}$ to \eqref{eq:sp*} coincides with 
the restrictions of $p_{n}^{(i_{1},\dots,i_{n})} \otimes \mathrm{id}_{J_{i_{n+1}}}$ and of $\mathrm{id}_{J_{i_{1}}} \otimes p_{n}^{(i_{2},\dots,i_{n+1})}$ to the same space. 
This proves the claimed commutativity. 

In case $(ii)$, the $(i_{1}, \dots, i_{n+1})$-specialization of \eqref{eq:si(n+1)} simplifies to give 
\[     p_{i_{1}, \dots, i_{n+1}} \circ (\Delta_{2} \otimes \mathrm{id}_{J}^{\otimes (n-1)}) \circ \Delta_{n}|_{J_{i}} = p_{i_{1}, \dots, i_{n+1}} \circ (\mathrm{id}_{J} \otimes \Delta_{n}) \circ \Delta_{2}|_{J_{i}},     \]
for all the other terms trivially vanish. 
This identity is equivalent to the commutativity of the square 
\[
\xymatrix
{
J_{i_{1} + \dots + i_{n+1}+2-n}
\ar[rr]^{\iota_{i_{1},i_{2} + \dots + i_{n+1}+2-n}}
\ar[dd]^{\iota_{i_{1} +i_{2}, i_{3}, \dots, i_{n+1}}}
&
&
J_{i_{1}} \otimes J_{i_{2} + \dots + i_{n+1}+2-n}
\ar[dd]^{\mathrm{id}_{J_{i_{1}}} \otimes \iota_{i_{2}, \dots, i_{n+1}}}
\\
&
&
\\
J_{i_{1}+i_{2}} \otimes J_{i_{3}} \otimes \dots \otimes J_{i_{n+1}}
\ar[rr]^{\iota_{i_{1}, i_{2}} \otimes \mathrm{id}_{J_{i_{3}}} \otimes \dots \otimes \mathrm{id}_{J_{i_{n+1}}}}
&
&
J_{i_{1}} \otimes \dots \otimes J_{i_{n+1}}
}
\]
The assumption on the integers $i_{1}, \dots, i_{n+1}$ tells us that the horizontal maps are the plain inclusions. 
The commutativity follows directly from the definitions, but we provide a proof. 
Consider the direct sum decomposition 
\begin{equation}
\label{eq:sp**}
\begin{split}
     &\bigoplus_{N \in \NN_{\geq 2}} \Big(\big(\bigcap_{\bar{m} \in \mathrm{Par}_{n+2}(N)} 
V^{(m_{1})} \otimes J_{i_{1}} \otimes V^{(m_{2})} \otimes J_{i_{2}-1} \otimes 
\dots 
\otimes J_{i_{n+1}-1} \otimes V^{(m_{n+2})}\big) 
\\
&\cap \big(\bigcap_{\bar{m} \in \mathrm{Par}_{n+1}(N)} V^{(m_{1})} \otimes J_{i_{1} + i_{2} - 1} \otimes V^{(m_{2})} \otimes J_{i_{3}-1} 
\otimes \dots 
\otimes J_{i_{n+1}-1} \otimes V^{(m_{n+1})} \big)\Big).  
\end{split}
\end{equation}
It contains $J_{i_{1} + \dots + i_{n+1}+2-n}$. 
On the one hand, Proposition \ref{proposition:imp} and Lemma \ref{lemma:interj} tell us that the $N$-th direct summand of the former is trivially included in the $N$-th direct summand of  
\[     \bigoplus_{N \in \NN_{\geq 2}} J_{i_{1}} \otimes \big(\bigcap_{\bar{m} \in \mathrm{Par}_{n+1}(N)} 
V^{(m_{1})} \otimes J_{i_{2}-1} \otimes V^{(m_{2})} \otimes \dots \otimes V^{(m_{n})} \otimes J_{i_{n+1}-1} \otimes V^{(m_{n+1})}\big),      \] 
for all $N \in \NN_{\geq 2}$. 
On the other hand, it is obvious that the $N$-th direct summand of \eqref{eq:sp**} is contained in the $N$-th direct summand of
\[     \bigoplus_{N \in \NN_{\geq 2}} \big(\bigcap_{\bar{m} \in \mathrm{Par}_{n+1}(N)} V^{(m_{1})} \otimes J_{i_{1} + i_{2} - 1} \otimes V^{(m_{2})} \otimes J_{i_{3}-1} 
\otimes \dots 
\otimes J_{i_{n+1}-1} \otimes V^{(m_{n+1})} \big),      \] 
for all $N \in \NN_{\geq 2}$. 
It is trivial to see that each independent family given by the $N$-th direct summands of each of the two previous direct sum decompositions is a subfamily of the independent family 
\[     \{ \bigcap_{\bar{m} \in \mathrm{Par}_{n}(N)} J_{i_{1}} \otimes J_{i_{2} - 1} \otimes V^{(m_{1})} \otimes J_{i_{3}-1} \otimes V^{(m_{2})} 
\otimes \dots 
\otimes J_{i_{n+1}-1} \otimes V^{(m_{n})} \}_{N \in \NN_{\geq 2}}.     \]
The comments at the end of the paragraph before Remark \ref{remark:util} tell us that the restriction of $\mathrm{id}_{J_{i_{1}}} \otimes p_{n}^{(i_{2},\dots,i_{n+1})}$ 
to \eqref{eq:sp**} coincides with the restrictions of $p_{n}^{(i_{1}+i_{2},i_{3},\dots,i_{n+1})}$ to the same space. 
This proves the claim. 

Dually to $(ii)$, under the assumptions of $(iii)$, the $(i_{1}, \dots, i_{n+1})$-specialization of \eqref{eq:si(n+1)} simplifies as  
\[     p_{i_{1}, \dots, i_{n+1}} \circ (\mathrm{id}_{J}^{\otimes (n-1)} \otimes \Delta_{2}) \circ \Delta_{n}|_{J_{i}} = p_{i_{1}, \dots, i_{n+1}} \circ (\Delta_{n} \otimes \mathrm{id}_{J}) \circ \Delta_{2}|_{J_{i}},     \]
for all the other terms are easily seen to vanish. 
This identity is equivalent to the commutativity of the square 
\[
\xymatrix
{
J_{i_{1} + \dots + i_{n+1}+2-n}
\ar[rr]^{\iota_{i_{1}+ \dots + i_{n}+2-n,i_{n+1}}}
\ar[dd]^{\iota_{i_{1}, \dots, i_{n-1}, i_{n} + i_{n+1}}}
&
&
J_{i_{1} + \dots + i_{n}+2-n} \otimes J_{i_{n+1}}
\ar[dd]^{\iota_{i_{1}, \dots, i_{n}} \otimes \mathrm{id}_{J_{i_{n+1}}}}
\\
&
&
\\
J_{i_{1}} \otimes \dots \otimes J_{i_{n-1}} \otimes J_{i_{n}+i_{n+1}} 
\ar[rr]^{\mathrm{id}_{J_{i_{1}}} \otimes \dots \otimes \mathrm{id}_{J_{i_{n-1}}} \otimes \iota_{i_{n}, i_{n+1}}}
&
&
J_{i_{1}} \otimes \dots \otimes J_{i_{n+1}}
}
\]
Notice that the hypotheses on the integers $i_{1}, \dots, i_{n+1}$ imply that the horizontal maps are the plain inclusions. 
The commutativity is clear from the definitions and follows by an analogous argument to the one given in the previous paragraph. 

Finally, in the case $(iv)$,  the $(i_{1}, \dots, i_{n+1})$-specialization of \eqref{eq:si(n+1)} gives  
\begin{multline*}
     p_{i_{1}, \dots, i_{n+1}} \circ (\mathrm{id}_{J}^{\otimes (j-2)} \otimes \Delta_{2} \otimes \mathrm{id}_{J}^{\otimes (n-j+1)}) \circ \Delta_{n}|_{J_{i}} 
\\ 
= p_{i_{1}, \dots, i_{n+1}} \circ (\mathrm{id}_{J}^{\otimes (j-1)} \otimes \Delta_{2} \otimes \mathrm{id}_{J}^{\otimes (n-j)}) \circ \Delta_{n}|_{J_{i}},     
\end{multline*}
for all the other terms are easily seen to vanish. 
This identity is equivalent to 
\begin{multline}    
\label{eq:ult}
(\mathrm{id}_{J_{i_{1}}} \otimes \dots \otimes \mathrm{id}_{J_{i_{j-2}}} \otimes \iota_{i_{j-1},i_{j}} \otimes \mathrm{id}_{J_{i_{j+1}}} \otimes \dots \otimes \mathrm{id}_{J_{i_{n+1}}}) \circ \iota_{i_{1}, \dots, i_{j-2},i_{j-1}+i_{j},i_{j+1}, \dots, i_{n+1}}
\\ 
= (\mathrm{id}_{J_{i_{1}}} \otimes \dots \otimes \mathrm{id}_{J_{i_{j-1}}} \otimes \iota_{i_{j},i_{j+1}} \otimes \mathrm{id}_{J_{i_{j+2}}} \otimes \dots \otimes \mathrm{id}_{J_{i_{n+1}}}) \circ \iota_{i_{1}, \dots, i_{j-1},i_{j}+i_{j+1},i_{j+2}, \dots, i_{n+1}},     
\end{multline}
By the assumption on the integers $i_{1}, \dots, i_{n+1}$ we have that the left maps of the left and right compositions are inclusions. 
This identity is proved as follows. 
Consider the direct sum decomposition 
\begin{equation}
\label{eq:sp***}
\begin{split}
     &\bigoplus_{N \in \NN_{\geq 2}} \big(\bigcap_{\bar{m} \in \mathrm{Par}_{n}(N-2)} 
V^{(m_{1})} \otimes J_{i_{1}-1} \otimes V^{(m_{2})} \otimes \dots \otimes J_{i_{j-2}-1} \otimes V^{(m_{j-1})} \otimes J_{i_{j-1}}   
\\
&\otimes J_{i_{j}} \otimes J_{i_{j+1}} \otimes V^{(m_{j})} \otimes J_{i_{j+2}-1} \otimes V^{(m_{j+1})} \otimes 
\dots \otimes V^{(m_{n+1})} \otimes J_{i_{n+1}-1} \otimes V^{(m_{n})}\big), 
\end{split}
\end{equation}
which trivially includes $J_{i_{1} + \dots + i_{n+1} + 2 - n}$. 
Furthermore, its $N$-th direct summand includes the $N$-th direct summand of the decomposition \eqref{eq:sp} corresponding to $\bar{j} = (i_{1}, \dots, i_{j-2}, i_{j-1}+i_{j},i_{j+1},\dots, i_{n+1})$ 
and $\bar{j} = (i_{1}, \dots, i_{j-1}, i_{j}+i_{j+1},i_{j+2},\dots, i_{n+1})$, for all $N \in \NN_{\geq 2}$. 
So, the restriction of the projection of this decomposition onto the $n$-th direct summand to the space \eqref{eq:sp} for $\bar{j} = (i_{1}, \dots, i_{j-2}, i_{j-1}+i_{j},i_{j+1},\dots, i_{n+1})$ 
coincides with $p_{n}^{\bar{j}}$, and the same applies if $\bar{j} = (i_{1}, \dots, i_{j-1}, i_{j}+i_{j+1},i_{j+2},\dots, i_{n+1})$. 
This proves identity \eqref{eq:ult}. 
 
The properties satisfied by the counit and the augmentation are also clear. 
The proposition is thus proved. 
\end{proof}

\begin{remark}
Note that the pattern of the proof is, roughly, (the dual of) the one given by \cite{HL}, Thm. 4.5, even though our case is much involved. 
This similarity is for us, however, another indication that this definition of multi-Koszul property resembles the notion of generalized Koszul algebras given by Berger. 
\end{remark}

\begin{remark}
It is easy to check that the $A_{\infty}$-algebra structure for $J^{\#}$ given in the proposition coincides with the one given in \cite{HeRe}, Rmk. 3.25, in the particular case that $A$ is 
finitely generated in degree $1$ and has a finite dimensional space of relations. 
\end{remark}

The following definitions are the duals to the ones given by Kadeishvili in \cite{K80}, p. 235, and we refer to \cite{Prou}, Chapitre 3, Section 5, even though the sign convention is different.  
We recall that a linear map $\tau : C \rightarrow A$ of cohomological degree $1$ and Adams degree zero from a coaugmented Adams graded $A_{\infty}$-coalgebra $C$ 
to an augmented Adams graded dg algebra $A$ is called a \emph{twisting cochain} (sometimes called \emph{homotopical or generalized twisting cochain}) 
if the composition of $\tau$ with the augmentation of $A$ vanishes, the composition of the unit of $C$ with $\tau$ is also zero, and if 
\[     d \circ \tau = \sum_{n \in \NN} \mu^{(n)} \circ \tau^{\otimes n} \circ \Delta_{n},     \]
where $d$ is the differential of $A$ and $\mu^{(n)}$ denotes the $(n-1)$-th iteration of the product of $A$, as explained for the tensor algebra at the beginning of Subsection \ref{subsec:inter}. 
Note that the previous sum is finite when evaluated at $c \in C$, for the family $\{ \Delta_{i} \}_{i \in \NN}$ is locally finite. 
We denote by $B^{+}(\place)$ the bar construction of an augmented $A_{\infty}$-algebra and by $\Omega^{+}(\place)$ the cobar construction of an augmented $A_{\infty}$-coalgebra. 
Moreover, $\tau_{C} : C \rightarrow \Omega^{+}(C)$ indicates the \emph{universal twisting cochain}, given by the composition of the projection $C \rightarrow C/\mathrm{Im}(\eta)$, the shift 
and the canonical inclusion, where $\eta$ denotes the coaugmentation of $C$ (see \cite{Prou}, D\'ef. 3.14). 
Dually, for an augmented Adams graded dg algebra $A$, its \emph{universal twisting cochain} is the map $\tau_{A} : B^{+}(A) \rightarrow A$, given by the composition of the minus canonical 
projection onto $\mathrm{Ker}(\epsilon)$, the shift and canonical inclusion, where $\epsilon$ denotes the augmentation of $A$ (see \cite{Prou}, (2.23)). 
Then, twisting cochains $\tau : C \rightarrow A$ are in correspondence with morphisms of dg algebras $f_{\tau} : \Omega^{+}(C) \rightarrow A$ via $f_{\tau} \circ \tau_{C} = \tau$ 
(see \cite{Prou}, Lemme 3.17). 
In this case we may consider the \emph{twisted tensor product} $C \otimes_{\tau} A$, whose underlying vector space is given by the usual tensor product. 
We will consider it only in the case $A$ has vanishing differential. 
Under this assumption the twisted tensor product is given by the complex of graded right $A$-modules for the regular right action of $A$ provided with a differential $d_{\tau}$ defined as 
\[     \sum_{n \in \NN} (\mathrm{id}_{C} \otimes \mu^{(n)}) \circ (\mathrm{id}_{C} \otimes \tau^{\otimes (n-1)} \otimes \mathrm{id}_{A}) \circ (\Delta_{n} \otimes \mathrm{id}_{A}).     \] 

The following theorem must be well-known to the experts, 
but we sketch a proof. 
It was announced by B. Keller at the X ICRA of Toronto, Canada, in 2002.
\begin{theorem}
\label{theorem:keller}
Let $C$ be a minimal Adams graded coaugmented $A_{\infty}$-coalgebra and $A$ be a nonnegatively graded connected algebra. 
Then, the following are equivalent:
\begin{itemize}
\item[(i)] There is a quasi-isomorphism of Adams graded augmented minimal $A_{\infty}$-algebras 
\[     \mathcal{E}xt_{A}^{\bullet}(k,k) \rightarrow C^{\#}.     \]
\item[(ii)] There is a twisting cochain $\tau : C \rightarrow A$ such that the twisted tensor product $C \otimes_{\tau} A$ is a minimal projective resolution of the trivial right $A$-module $k$.
\end{itemize}
\end{theorem}
\begin{proof}
Let us show that $(ii) \Rightarrow (i)$. 
First note that we have a morphism of augmented dg algebras $f_{\tau} : \Omega^{+}(C) \rightarrow A$, which in turn induces a morphism of augmented dg coalgebras $B^{+}(f_{\tau}) : B^{+}(\Omega^{+}(C)) \rightarrow B^{+}(A)$. 
By \cite{Prou}, Thm. 3.25, we know that $f_{\tau}$ is a quasi-isomorphism if and only if $B^{+}(f_{\tau})$ is so, if and only if $B^{+}(\Omega^{+}(C)) \otimes_{f_{\tau} \circ \tau_{\Omega^{+}(C)}} A$ is acyclic (in the sense given by \cite{Prou}, D\'ef. 1.26). 
Using the adjunction quasi-isomorphism $\psi : C \rightarrow B^{+}(\Omega^{+}(C))$ (see \cite{Prou}, Thm. 2.28, or \cite{LH}, Lemme 1.3.2.3) we get that $C \otimes_{\tau} A$ is quasi-isomorphic 
to $B^{+}(\Omega^{+}(C)) \otimes_{f_{\tau} \circ \tau_{\Omega^{+}(C)}} A$ (see \cite{Prou}, Thm. 3.23 for the existence of the map. 
The fact that it is a quasi-isomorphism follows easily, see for instance \cite{Prou}, Thm. 2.8, (i)). 
Hence, the assumption that $C \otimes_{\tau} A$ is acyclic yields that the map $B^{+}(f_{\tau}) : B^{+}(\Omega^{+}(C)) \rightarrow B^{+}(A)$ 
is a quasi-isomorphism of coaugmented dg coalgebras. 
If we consider the composition of the adjunction quasi-isomorphism $\psi : C \rightarrow B^{+}(\Omega^{+}(C))$ of coaugmented $A_{\infty}$-coalgebras with the quasi-isomorphism 
$B^{+}(f_{\tau})$ we get thus a quasi-isomorphism of coaugmented $A_{\infty}$-coalgebras $C \rightarrow B^{+}(A)$, which in turn induces a quasi-isomorphism of augmented 
$A_{\infty}$-algebras $B^{+}(A)^{\#} \rightarrow C^{\#}$. 
Using the quasi-isomorphism $\mathcal{E}xt_{A}^{\bullet}(k,k) \rightarrow B^{+}(A)^{\#}$ of augmented $A_{\infty}$-algebras given by the theorem of Kadeishvili, 
we get the desired quasi-isomorphism of augmented $A_{\infty}$-algebras. 
It further preserves the Adams degree for all the involved maps preserve it. 

We will now prove the converse. 
Suppose thus that $f : \mathcal{E}xt_{A}^{\bullet}(k,k) \rightarrow C^{\#}$ is a quasi-isomorphism of minimal Adams graded augmented $A_{\infty}$-algebras. 
This implies that we have a quasi-isomorphism of augmented $A_{\infty}$-algebras $g : B^{+}(A)^{\#} \rightarrow C^{\#}$. 
Define $\tau$ as the twisting cochain composition of $g^{\#}$ and the universal twisting cochain $\tau_{A} : B^{+}(A) \rightarrow A$ of $A$ (see \cite{Prou}, D\'ef. 3.21 and Lemme 3.22).   
Since $g^{\#}$ is a quasi-isomorphism, we get that the map $g^{\#} \otimes \mathrm{id}_{A}$ defines a quasi-isomorphism from $C \otimes_{\tau} A$ to $B^{+}(A) \otimes_{\tau_{A}} A$, 
which is always acyclic (see for example \cite{Prou}, Thm. 3.19).  
See \cite{Prou}, Thm. 3.23 for the proof that this map is a morphism of complexes of $A$-modules. 
The fact that it is a quasi-isomorphism follows easily (see for instance \cite{Prou}, Thm. 2.8, (i)). 
Hence $C \otimes_{\tau} A$ is also acyclic. 
\end{proof}

For the coaugmented $A_{\infty}$-coalgebra structure defined on $J$, we now set $\tau : J \rightarrow A$ as the linear map given by the composition of the canonical projection 
$J \rightarrow J_{1} = V$ and the inclusion $V \subseteq A$. 
It is easy to check that $\tau$ is a twisting cochain, and furthermore $J \otimes_{\tau} A$ coincides with the right multi-Koszul complex of $A$. 
Indeed, we first note that the underlying homological bigraded vector space of $J \otimes_{\tau} A$ coincides with the corresponding one of $K(A)'_{\bullet}$. 
Moreover, the differential $d_{\tau}$ restricted to $J_{i} \otimes A$, for $i$ odd, is trivially seen to coincide with the differential $b_{i}'$ of $K(A)'_{\bullet}$. 
The case $i$ even follows from Remark \ref{remark:util}. 
Hence, by Proposition \ref{proposition:yonedaainf} and the previous theorem we have the following result. 
\begin{theorem}
\label{theorem:ainftres}
Let $A$ be a multi-Koszul algebra, and let $\{ J_{i} \}_{i \in \NN_{0}}$ be the collection of graded vector subspaces of $T(V)$ defined by $J_{0} = k$ and the recursive identities 
\eqref{eq:j_impar} and \eqref{eq:j_par}. 
Then the Adams graded augmented $A_{\infty}$-algebras $\mathcal{E}xt_{A}^{\bullet}(k,k)$ and $J^{\#}$ are quasi-isomorphic, where the structure of 
Adams graded augmented $A_{\infty}$-algebra of $J^{\#}$ was given in Proposition \ref{proposition:yonedaainf}. 
\end{theorem}

\bibliographystyle{model1-num-names}
\addcontentsline{toc}{section}{References}



\end{document}